\documentclass{article}

\usepackage{amsmath,amsthm,amsfonts,amssymb}
\usepackage{tikz}
\usepackage{color}
\usepackage{graphicx}
\usepackage{subcaption} 
\usepackage{comment}
\usepackage{enumerate}

\DeclareMathOperator{\soc}{soc}
\DeclareMathOperator{\Aut}{Aut}
\DeclareMathOperator{\BiCay}{BiCay}

\DeclareMathOperator{\Cos}{Cos}
\DeclareMathOperator{\Diag}{Diag}
\DeclareMathOperator{\PSL}{PSL}
\DeclareMathOperator{\Sym}{Sym}
\DeclareMathOperator{\OG}{\mathcal{OG}}

\DeclareMathOperator{\B}{\mathcal{B}}

\theoremstyle{plain}
\newtheorem{Theorem}{Theorem}
\newtheorem{Corollary}[Theorem]{Corollary}
\newtheorem{Proposition}[Theorem]{Proposition}

\newtheorem{Construction}{Construction}

\theoremstyle{definition}
\newtheorem{remark}{Remark}
\newtheorem{method}{Method}

\newtheorem{Lemma}[Theorem]{Lemma}

\title{Four-Valent Oriented Graphs of Biquasiprimitive Type}
\author{Nemanja Poznanovi\'{c}\footnote{University of Melbourne, Australia.  email: {nempoznanovic@gmail.com}},\hspace{1cm} Cheryl E. Praeger\footnote{University of Western Australia, Australia. email:  {cheryl.praeger@uwa.edu.au}}}

\begin{document}
	\maketitle

\begin{abstract}
		Let $\OG(4)$ denote the family of all graph-group pairs $(\Gamma,G)$ where $\Gamma$ is 4-valent, connected and $G$-oriented  ($G$-half-arc-transitive).	Using a novel application of the structure theorem for biquasiprimitive permutation groups of the second author, we produce a description of all pairs $(\Gamma, G) \in\OG(4)$ for which every nontrivial normal subgroup of $G$ has at most two orbits on the vertices of $\Gamma$. In particular we show that $G$ has a unique minimal normal subgroup $N$ and that  $N \cong T^k$ for a simple group $T$ and $k\in \{1,2,4,8\}$.
		This provides a crucial step towards a general description of the long-studied family $\OG(4)$ in terms of a normal quotient reduction. We also give several methods for constructing pairs $(\Gamma, G)$ of this type and provide many new infinite families of examples, covering each of the possible structures of the normal subgroup $N$.
\end{abstract}

\section{Introduction}
All graphs considered in this paper are simple, undirected and finite. A graph $\Gamma$ is said to be \textit{$G$-oriented} (or $G$-half-arc-transitive) with respect to some group $G \leq \Aut(\Gamma)$ if $G$ acts transitively on the vertices and edges of $\Gamma$, but $G$ is not transitive on the arcs. In this case, $G$ has exactly two orbits on the arcs, and for each arc-orbit $\Delta$ and each edge $\{\alpha, \beta\}$, $\Delta$ contains exactly one of the arcs $(\alpha, \beta)$ or $(\beta, \alpha)$. Thus $\Delta$ is a $G$-invariant orientation of the edge set of $\Gamma$.

 Every $G$-oriented graph necessarily has even valency and all connected components of a $G$-oriented graph are pairwise isomorphic $G$-oriented graphs. It is thus natural to restrict attention to $G$-oriented graphs which are connected. For each even integer $m\geq2$, we let $\OG(m)$ denote the family of graph-group pairs $(\Gamma, G)$ where $\Gamma$ is connected, $m$-valent and $G$-oriented.

It is easy to see that the family $\OG(2)$ consists only of oriented cycle graphs. On the other hand, study of the family $\OG(4)$ has been an active area of research for several decades and has taken a number of different directions (especially because of its connection with the embedding of graphs into Riemann surfaces). For a good summary of this research up to 1998 see \cite{maruvsivc1998recent}, for a more recent overview see \cite[Section 2]{al2015finite}.

A particularly useful tool for studying $\OG(4)$ was given in \cite{maruvsivc1998half}, where several important combinatorial parameters were defined for graphs in this family based on certain cyclic subgraphs called \textit{$G$-alternating cycles}. This lead to the formulation of an approach to studying $\OG(4)$ by considering various quotients defined in terms of the $G$-alternating cycles, see \cite{maruvsivc1999tetravalent}. 

The combined results of \cite{maruvsivc1998half,maruvsivc1999tetravalent} 
provide a complete classification of some subfamilies of $\OG(4)$ and prove that  pairs $(\Gamma, G)\in\OG(4)$ not contained in these subfamilies are covers of other members of $\OG(4)$ satisfying certain combinatorial conditions. In particular, this approach naturally identifies two subfamilies of $\OG(4)$ as `alternating-cycle-basic' in the sense that all 4-valent $G$-oriented graphs (other than those already classified) are covers of these basic members. However the analysis in \cite{maruvsivc1998half,maruvsivc1999tetravalent} provided no tools for studying the `basic' graphs relative to this reduction.


More recently, a new framework for studying the family $\OG(4)$ was proposed in \cite{al2015finite} and developed further in \cite{al2016normal,al2017finite}. This new approach aims to analyse $\OG(4)$ using a normal quotient reduction, a method which has been successfully used to study other families of graphs with prescribed symmetry conditions, see for instance \cite{morris2009strongly, praeger1993nan, praeger1999finite}, but has never been applied to oriented graphs. The
aim of this approach (explained in detail below) is to describe the family $\OG(4)$ in terms of graph quotients arising from normal subgroups of the groups contained in this family. In particular, it is again possible to identify three subfamilies of $\OG(4)$ which are `normal-quotient-basic' in the sense that all pairs $(\Gamma, G)\in \OG(4)$ are \textit{normal} covers of at least one of these basic pairs (see Section \ref{ssNormal}).

It is likely that these two approaches may converge in a significant proportion of cases. The quotient graph $\Gamma_{\B}$ constructed in \cite{maruvsivc1999tetravalent} from a given pair $(\Gamma, G) \in \OG(4)$, related to the $G$-alternating cycles, has been studied again recently by Ramos Rivera and \v{S}parl \cite[Construction 5.4]{rivera2019new}. Provided a mild condition on parameters is satisfied (the attachment number should be less than the radius), they prove that $\Gamma_{\B}$ is a normal quotient  \cite[Theorem 5.6]{rivera2019new} and hence may be studied using the powerful theory developed in \cite{al2016normal,al2017finite,al2015finite}, supplemented by the results of this paper.

In this paper we answer \cite[Problem 1.2]{al2015finite} and provide a description of the pairs $(\Gamma, G)\in \OG(4)$ of biquasiprimitive type (one of the three families of pairs defined to be `basic' with respect to normal quotients). Our solution provides an important step towards a description of $\OG(4)$ in terms of normal quotients. For a detailed description of this programme and definitions of all basic pairs see Section \ref{ssNormal}.

\bigskip

\noindent\textbf{Biquasiprimitive Basic Pairs.}  If ($\Gamma, G)\in \OG(4)$ is basic of biquasiprimitive type then the group $G$ contains a normal subgroup $N$ with exactly two orbits on the vertices of $\Gamma$, and all nontrivial normal subgroups of $G$ have at most two orbits.
It is easy to see that $\Gamma$ is bipartite: since $\Gamma$ is connected there is an edge joining vertices in different $N$-orbits, and since $G$ normalizes $N$ and $\Gamma$ is $G$-edge-transitive, each edge joins vertices in different $N$-orbits. Thus the two orbits of $N$ form a bipartition of $\Gamma$.

It follows that there is an index two subgroup $G^+$ of $G$ which fixes the two parts of the bipartition of $\Gamma$ setwise. The main result of this paper is the following theorem which describes the biquasiprimitive basic pairs $(\Gamma, G)\in \OG(4)$ in a manner analogous to \cite[Theorem 1.3]{al2015finite} for the quasiprimitive case.

\begin{Theorem}\label{MainResult}
Suppose that $(\Gamma, G)\in \OG(4)$ is basic of biquasiprimitive type.  Then $G$ has a unique minimal normal subgroup  $N= \soc(G)$ contained in a unique intransitive index $2$ subgroup $G^+ \leq G$. Furthermore, $N \cong T^k$ where $T$ is a finite simple group and exactly one of the following holds:
\begin{enumerate}[(a)]
	\item $T$ is abelian and $k\leq 2$, or
	\item  $T$ is nonabelian,  $k \in \{1,2,4\}$, and $N$ is the unique minimal normal subgroup of $G^+$, or
	\item $T$ is nonabelian, $k= 2\ell$ with $\ell \in \{1,2,4\}$, $G^+$ has exactly two minimal normal subgroups each isomorphic to $T^\ell$, and $N$ is the direct product of these two subgroups.
\end{enumerate}
Moreover, there are infinitely many biquasiprimitive basic pairs $(\Gamma, G)\in \OG(4)$ described by each of the cases (a) - (c) and each value of $k$ in each case. 
\end{Theorem}


The first part of this paper establishes that cases (a)-(c) of Theorem \ref{MainResult} must hold in two steps. In Section \ref{SecTwoTypes} we  show that if $(\Gamma, G)\in \OG(4)$ is basic of biquasiprimitive type then $G$ has a unique minimal normal subgroup $N$ and one of the three cases (a) (b) or (c) holds for some $k \geq 1$. For this we use the structure theorem for biquasiprimitive groups given in \cite{praeger2003finite}. Then in Section \ref{SecRestricting} we use combinatorial arguments to obtain the various possibilities for the value of $k$ (the number of simple direct factors of $\soc(G)$) in each case.  


In the second part of this paper (Section \ref{SecConstructions}), we provide methods for constructing biquasiprimitive basic pairs and use them to construct various families of examples. In particular, we provide an infinite family of basic pairs for each of the cases described in Theorem \ref{MainResult}, for each possible value of $k$, thus proving the final assertion of Theorem \ref{MainResult}.

\section{Preliminaries}
 Unless otherwise stated we will let $V\Gamma$, $E\Gamma$ and $A\Gamma$ denote the vertex-, edge-, and arc-set of a given graph $\Gamma$ (an \textit{arc} is an ordered pair of adjacent vertices). Given a vertex $\alpha \in V\Gamma$ we let $\Gamma(\alpha)$ denote the neighbourhood of $\alpha$ in $\Gamma$. For fundamental graph-theoretic concepts we refer the reader to \cite{godsil200AGT}, and for group-theoretic concepts not defined here, please refer to \cite{praeger2018permutation}.

Given a group $G$ acting on a set $X$, we will always let $G^X$ denote the subgroup of $\Sym(X)$ induced by the group $G$. Given a permutation $g\in G$ and an element $x\in X$ we will let $x^g$ denote the image of $x$ under $g$. A permutation group $G^X$ is said to be \textit{semiregular} if only the identity element of $G$ fixes a point in $X$, and is said to be \textit{regular} if it is semiregular and transitive.

\subsection{G-oriented Graphs}
If $\Gamma$ is a $G$-oriented graph then the group $G$ is transitive on the vertices and edges but not on the arcs of $\Gamma$. It follows that the group $G$ has two orbits on the arc set of $\Gamma$ and these two orbits are paired. (Every arc $(u,v)$ in one orbit will have its reverse arc $(v,u)$ in the other orbit.) Either of these two $G$-orbits on the arc set of $\Gamma$ naturally give rise to a $G$-invariant orientation of the edges of $\Gamma$: simply take any arc $(u,v)$ of $\Gamma$ and then orient each edge $\{x,y\}$ from $x$ to $y$ if and only if $(u,v)^g = (x,y)$ for some $g\in G$. 

Given a pair $(\Gamma, G)\in \OG(4)$, a vertex $v_0 \in V\Gamma$ and any $G$-invariant orientation of $E\Gamma$, it will always be the case that two of the four edges incident to $v_0$ are oriented from $v_0$ to one of its neighbours, while the other two edges are oriented from a neighbour to $v_0$. In particular, the stabilizer $G_{v_0}$ of a vertex $v_0 \in V\Gamma$ will always have two orbits of length two on the neighbourhood of $v_0$, and we can think of these two orbits as the in-neighbours and out-neighbours of $v_0$ with respect to a given orientation. We denote these two sets by $\Gamma_{in}(v_0)$ and $\Gamma_{out}(v_0)$ respectively.

Given a connected, 4-valent, $G$-vertex-transitive graph $\Gamma$, we may show that $(\Gamma, G) \in \OG(4)$ by showing that $G_{v_0}$ has two orbits of size 2 on $\Gamma(\alpha)$, and that no automorphism of $G$ can reverse an edge of $\Gamma$.

An \textit{oriented $s$-arc} of a $G$-oriented graph is a sequence of vertices $(v_0, v_1,...,v_s)$ of $\Gamma$, such that for each $i \in \{0,...,s-1\}$, $v_i$ and $v_{i+1}$ are adjacent, and each edge $\{v_i, v_{i+1}\}$ is oriented from $v_i$ to $v_{i+1}$.
We will make use of the following important fact concerning oriented $s$-arcs of $G$-oriented graphs, the proof of which can be found in the first part of the proof of \cite[Lemma 6.2]{al2015finite}.

	\begin{Lemma}\label{regular_s_arcs}
	Let $(\Gamma, G) \in \OG(4)$ and let $s\geq 1$ be the largest integer such that $G$ acts transitively on the oriented $s$-arcs of $\Gamma$. Then $G$ acts regularly on the oriented $s$-arcs of $\Gamma$.
\end{Lemma}
Now let $(\Gamma, G) \in \OG(4)$ and take a vertex $\alpha \in V\Gamma$. Let $s$ be as in the statement of Lemma \ref{regular_s_arcs} and consider an  oriented $s$-arc $(v_0,v_1,...,v_s)$ of $\Gamma$ where $\alpha := v_0$. Since $G$ is regular on the oriented $s$-arcs of $\Gamma$, it follows that the vertex-stabilizer $G_{v_0}$ is regular on the oriented $s$-arcs starting at $v_0$. From this  it follows that for each $i$ with $0\leq i \leq s$, the subgroup $G_{{v_0},...,{v_{s-i}}}$ has order $2^i$. In particular, $|G_{{v_0},...,{v_{s-1}}}| = 2$, and the stabilizer of a vertex $G_{v_0} = G_\alpha$ is a 2-group.

\subsection{Normal Quotients} \label{ssNormal}	

	Given a pair $(\Gamma, G) \in \OG(4)$ and a normal subgroup $N$ of $G$, we define a new graph $\Gamma_N$ called a \textit{$G$-normal-quotient} of $\Gamma$. The vertices of $\Gamma_N$  are the $N$-orbits on the vertices of $\Gamma$, with an edge between two $N$-orbits $\{B, C\}$ in $\Gamma_N$ if and only if there is an edge of the form $\{\alpha, \beta\}$ in $\Gamma$, with $\alpha \in B$ and $\beta \in C$. 
	The group $G$ induces a group $G_N = G/K$ of automorphisms of $\Gamma_N$, where $K$ is the kernel of the $G$-action on $\Gamma_N$. By definition $N\leq K$, and hence the $K$-orbits are the same as the $N$-orbits so $\Gamma_K=\Gamma_N$. However $K$ may be strictly larger than $N$.   
	
	If $(\Gamma_N, G_N)$ is itself a member of $\OG(4)$, that is, $\Gamma_N$ is a 4-valent $G_N$-oriented graph, then $\Gamma$ is said to be a \textit{$G$-normal cover of $\Gamma_N$}. In general however, the pair $(\Gamma_N, G_N)$ need not lie in  $\OG(4)$, and the various possibilities for such normal quotient pairs $(\Gamma_N, G_N)$ were identified in \cite[Theorem 1.1]{al2015finite}. In particular, it was proved  that for any $(\Gamma, G) \in \OG(4)$, and any nontrivial normal subgroup $N$ of $G$, either $(\Gamma_N, G_N)$ is also in $\OG(4)$ and $\Gamma$ is a $G$-normal cover of $\Gamma_N$, or $\Gamma_N$ is isomorphic to $K_1$, $K_2$ or a cycle $C_r$, for some $r\geq 3$. 
	A pair $(\Gamma_N, G_N)$ where $\Gamma_N$ is isomorphic to one of $K_1$, $K_2$ or $C_r$ is defined to be \textit{degenerate}, while a pair ($\Gamma, G) \in \OG(4)$ for which $(\Gamma_N, G_N)$ is degenerate relative to every non-trivial normal subgroup $N$ of $G$ is defined to be \textit{basic}.  
	
	Since \cite[Theorem 1.1]{al2015finite} ensures that every member of $\OG(4)$ is a normal cover of a basic pair, this result suggests a framework for studying the family $\OG(4)$ using normal quotient reduction. The goal of this framework is to improve understanding of this family by developing a theory to describe the basic pairs in $\OG(4)$, and subsequently developing a theory to describe the $G$-normal covers of these basic pairs.
	
	Work in this direction was initiated in \cite{al2015finite} where the basic pairs were further divided into three types and the basic pairs of quasiprimitive type were analysed. A pair $(\Gamma, G) \in \OG(4)$ is said to be basic of \textit{quasiprimitive type} if all $G$-normal quotients $\Gamma_N$ of $\Gamma$ are isomorphic to $K_1$. This occurs precisely when all non-trivial normal subgroups of $G$ are transitive on the vertices of $\Gamma$. A permutation group with this property is said to be quasiprimitive, and there is a general structure theorem available for quasiprimitive groups  analogous to the O'nan-Scott Theorem for primitive permutation groups in \cite{praeger1993nan}.
	Using this tool, as well as combinatorial properties of the family $\OG(4)$, it was shown \cite[Theorem 1.3]{al2015finite} that if $(\Gamma, G) \in \OG(4)$ is basic of quasiprimitive type, then $G$ has a unique minimal normal subgroup $N \cong T^k$ where $T$ is a nonabelian finite simple group and $k\leq 2$.
	
	Of course, every pair $(\Gamma, G) \in \OG(4)$ will have at least one normal quotient $\Gamma_N$ isomorphic to $K_1$ since we may take the quotient with respect to the full group $G$. If the only normal quotients of a pair $(\Gamma, G) \in \OG(4)$ are the graphs $K_1$ or $K_2$, and $\Gamma$ has at least one $G$-normal quotient isomorphic to $K_2$, then $(\Gamma, G)$ is said to be basic of \textit{biquasiprimitive type}. The group $G$ here is biquasiprimitive: it is not quasiprimitive but each nontrivial normal subgroup has at most two orbits. Again, there is a structure theorem for biquasiprimitive groups available in \cite{praeger2003finite}.
	
	The basic pairs in $\OG(4)$ which are neither quasiprimitive nor biquasiprimitive must have at least one normal quotient isomorphic to a cycle graph $C_r$, and hence are said to be of \textit{cycle type.} Work towards describing the basic pairs of cycle type was initiated in \cite{al2016normal} where several important families of these graphs, which have already been discussed in the literature, were analysed from a normal quotient point of view. A more general analysis of these pairs was done in \cite{al2017finite}, however further work is required to understand this type.
	
	The above discussion outlining the three types of basic pairs $(\Gamma, G) \in \OG(4)$ is summarised in Table \ref{BasicTable}. This table also includes references to the papers where the corresponding basic pairs were previously studied.
	The objective of this paper is to describe the basic pairs $(\Gamma, G) \in \OG(4)$ of biquasiprimitive type, several families of which were constructed in \cite{PPMatrix}.
	\begin{table}
		
		\begin{tabular}{ l l l l }
			\hline
			Basic Type & Possible $\Gamma_N$ for $1\neq N \lhd G$ &  Conditions on $G$-action on vertices & Reference \\ 
			\hline
			Quasiprimitive & $K_1$ only & quasiprimitive & \cite{al2015finite}\\  
			Biquasiprimitive & $K_1$ and $K_2$ only & biquasiprimitive  &  -- \\
			Cycle & At least one $C_r $ $(r\geq 3)$ & at least one quotient action $D_{2r}$ or $\mathbb{Z}_r$& \cite{al2016normal,al2017finite}\\
			\hline
		\end{tabular}\caption{Types of Basic Pairs $(\Gamma, G) \in \OG(4)$.}\label{BasicTable}
	\end{table}
	
\subsection{Bi-Cayley Graphs.}\label{ssBiCay} A \textit{bi-Cayley graph} $\Gamma$ is a graph which admits a semiregular group of automorphisms $H$ with two orbits on the vertex set of $\Gamma$. 
These graphs are important for our purposes as for many of the pairs $(\Gamma, G)\in \OG(4)$ studied in this paper, the group $G$ will have a normal subgroup $N$ contained in $G^+$ and acting semiregularly on the two parts of the bipartition of $\Gamma$. In such cases, $\Gamma$ is a bi-Cayley graph.

	Every bi-Cayley graph of a group $H$ may be constructed in the following way. Let $R$ and $L$ be inverse-closed subsets of $H$ which do not contain the identity, and let $S$ be a subset of $H$. Define the graph $\Gamma = \BiCay(H, R, L, S)$ to be a graph whose vertex set is the union of the sets $H_0=\{h_0 : h \in H\}$ and $H_1=\{h_1 : h \in H\}$ (two copies of the group $H$), and whose edge set is the union of the \textit{right edges} $\{\{h_0, g_0\} : gh^{-1}\in R\}$, the \textit{left edges} $\{\{h_1, g_1\} : gh^{-1}\in L\}$, and the \textit{spokes} $\{\{h_0, g_1\} : gh^{-1}\in S\}$. Note that if $\Gamma$ is connected then $H$ is generated by $R\cup L \cup S$  (however the converse does not necessarily hold).
	The group $H$ then acts by right multiplication on the vertices of $\Gamma$, and this action is semiregular with two orbits $H_0$ and $H_1$. See for instance \cite{conder2016edge,zhou2016automorphisms}.

	

\section{Biquasiprimitive Basic Pairs: two types.} \label{SecTwoTypes}
Suppose now that $(\Gamma, G) \in \OG(4)$ is a basic pair of biquasiprimitive type and recall that this implies that $\Gamma$ is bipartite. Let $X$ denote the vertex set of $\Gamma$ with $\{\Delta, \Delta'\}$ the bipartition of $X$, and let $G^+$ be the index 2 subgroup of $G$ fixing the two biparts $\Delta$ and $\Delta'$ setwise. Since $\Gamma$ is $G$-vertex-transitive it follows that $G^+$ is transitive on both $\Delta$ and $\Delta'$. 

In this section we will begin working towards the proof of Theorem \ref{MainResult}. We start with a lemma about the intransitive normal subgroups of $G$.

\begin{Lemma}\label{G^+ faithful}
		Let $ (\Gamma, G) \in \OG(4)$ be basic of biquasiprimitive type, and let $X$ denote the vertex set of $\Gamma$. Let $G^+$ be the subgroup of $G$ of index two with orbits $\Delta, \Delta'$ (the biparts of $X$). Then
		\begin{enumerate}
			\item [(a)] $G^+$ is faithful on $\Delta$ $($and $\Delta')$, and
			\item[(b)] any non-trivial intransitive normal subgroup $N$ of $G$ must have the sets $\Delta$ and $ \Delta'$ as its two orbits on $X$. In particular, $N$ is contained in $G^+$.
		\end{enumerate}
\end{Lemma}
\begin{proof}
	(a). Let $K$ be the subgroup of $G^+$ fixing $\Delta$ pointwise and suppose that $K \neq 1$, and hence that $K$ acts non-trivially on $\Delta'$. If $g \in G\backslash G^+$ then $K^g$ is the pointwise stabilizer of $\Delta'$ in $G^+$, and hence $K \cap K^g = 1$, so $\langle K, K^g \rangle \cong K \times K^g$.
	
	Now since both $K$ and $K^g$ are normal in $G^+$, and since $g^2 \in G^+$ (because $|G:G^+| = 2$), it follows that  $(K \times K^g)^g = K \times K^g$, and so $K\times K^g$ is a normal subgroup of $G$ contained in $G^+$. Thus $ K \times K^g$ has two orbits $\Delta$ and $\Delta'$. But this implies that $K$ is transitive on $\Delta'$, which is impossible since for any $\alpha \in \Delta$ we have $K \leq G_\alpha$, and $G_\alpha$ is not transitive on $\Gamma(\alpha) \subset \Delta'$. Thus part (a) holds.
	
	(b). Since $|V\Gamma| \geq |\{\alpha\} \cup \Gamma(\alpha)| = 5$, it follows that $|N| \geq \frac{1}{2}|V\Gamma| >2$, hence $N\cap G^+ \neq 1$ since $|N:N\cap G^+|\leq 2$. Thus $N\cap G^+$ is a nontrivial intransitive normal subgroup of $G$ contained in $G^+$, so its orbits are $\Delta$ and $\Delta'$, and these must also be the orbits of the intransitive normal subgroup $N$.
\end{proof}

Next we introduce a convenient framework for investigating these graphs, based on the Imprimitive Wreath Embedding Theorem \cite[Theorem 5.5]{praeger2018permutation}
which identifies the vertex set $X$ with $\{ v_i\mid v\in V, i\in\{0,1\}\}$, and $G$ with a transitive subgroup of $\Sym(V)\wr\Sym(2)$ in its natural imprimitive action, so that  $\Delta = \{ v_0\mid v\in V\}$ and $\Delta'=\{ v_1\mid v\in V\}$.   Since $G$ is transitive, its subgroup $G^+$ induces transitive subgroups
$(G^+)^\Delta$ and $(G^+)^{\Delta'}$ on $\Delta$ and $\Delta'$, each of which we identify with a transitive subgroup of $\Sym(V)$.    

Let $\tau\in\Sym(V)\wr\Sym(2)$ generate the top group, that is, $\tau: v_\varepsilon \rightarrow v_{1-\varepsilon}$ for each $v\in V, \varepsilon\in\{0,1\}$,
and note that $\tau$ conjugates each element $(h_1, h_2)\in\Sym(V)\times\Sym(V)$ to its reverse $(h_2, h_1)$. For a group $H$, $y\in H$, and $\varphi\in\Aut(H)$, we denote by $\iota_y$ the inner automorphism of $H$ induced by $y$, that is  $h\mapsto y^{-1}hy$, and by  $ \Diag_\varphi(H\times H)=\{(h,h^\varphi)\mid h\in H\}$ the diagonal subgroup of $H\times H$ corresponding to $\varphi$.

\begin{Proposition}\label{iso}
	Let $ (\Gamma, G) \in \OG(4)$ be basic of biquasiprimitive type, and let $X$ denote the vertex set of $\Gamma$. Let $G^+$ be the subgroup of $G$ of index two with orbits $\Delta, \Delta'$ in $X$, and let $H$ be the permutation group induced by $G^+$ on $\Delta$. Let $\alpha\in\Delta$ and $\beta\in \Gamma(\alpha)\subseteq \Delta'$. Then replacing $G$ by a conjugate in $\Sym(X)$ if necessary, we may take $X=\{ v_i\mid v\in V, i\in\{0,1\}\}, \Delta, \Delta'$  and $\alpha=u_0$ as above, where $u\in V$, and we may identify $H$ with a transitive subgroup of $\Sym(V)$, such that 
	\begin{enumerate}
		\item[(a)] $G\leq H\wr\Sym(2)$, so $H = (G^+)^\Delta = (G^+)^{\Delta'}$; and
		\item[(b)] for some  $y\in H$ and $\varphi\in\Aut(H)$ with $\varphi^2= \iota_y$, we have  $G^+ = \Diag_\varphi(H\times H),$ and $G=\langle G^+, g\rangle$, where $g:=(y,1)\tau$, and $\beta= \alpha^g=(u^y)_1$.  Also $G_\alpha = G_\alpha^+ \cong H_\alpha$ is a $2$-group.
	\end{enumerate}	
\end{Proposition}

\begin{proof}
	The first assertion that we may choose the identification of $X, \Delta, \Delta'$ so that the transitive subgroups $(G^+)^\Delta$ and $(G^+)^{\Delta'}$ determine the same subgroup $H$ of $\Sym(V)$ follows from  the embedding theorem  \cite[Theorem 5.5]{praeger2018permutation}.  Thus $G\leq H\wr\Sym(2)$, and
	$G^+\leq H\times H$. By Lemma~\ref{G^+ faithful}, $G^+$ is a diagonal subgroup of $H\times H$, so   $G^+ = \Diag_\varphi(H\times H),$ for some
	$\varphi\in\Aut(H)$. 
	
	Since $G$ is transitive on $X$, there exists $g=(h_1, h_2)\tau\in G$ such that $\beta=\alpha^g$. Set $s:= (1, h_2)\in H\times H$. Then $s$ induces a graph isomorphism from $\Gamma$ to the graph $\Gamma^s$ with vertex set $X$ and arc set consisting of all pairs $(v_\varepsilon^s, w_{1-\varepsilon}^s) = (v_\varepsilon, (w^{h_2})_{1-\varepsilon})$, where $(v_\varepsilon, w_{1-\varepsilon})$ is an arc of $\Gamma$.  Moreover $(\Gamma^s, G^s)\in\OG(4)$,
	the group $G^s =\langle (\Diag_{\varphi}(H\times H))^s, g^s\rangle$, and we have 
	$(\Diag_\varphi(H\times H))^s = \Diag_{\varphi\iota_{h_2}}(H\times H)$ and 
	\[
	g^s = (1, h_2^{-1})(h_1,h_2)\tau (1,h_2) = (h_1h_2, 1)\tau.
	\]
	Set $y:= h_1h_2$. Then $g^s$ maps $\alpha^s$ to its out-neighbour $\beta^s$ in $\Gamma^s$, and we have  $\alpha^s = \alpha$, and
	$\beta^s = (\alpha^g)^s = (\alpha^{s})^{g^s} = \alpha^{g^s}  = (u_0)^{(y,1)\tau} = (u^y)_1$. 
	
	Now replace $\Gamma, G, g, \varphi, \alpha, \beta$ by  $\Gamma^s, G^s, g^s, \varphi\iota_{h_2}, \alpha, \beta^s$. Then all assertions are proved apart from the equality $\varphi^2=\iota_y$, which we now prove (for the new $\varphi)$.  Since $g=(y,1)\tau$ normalises $G^+ = \Diag_\varphi(H\times H)$, it follows that, for all $h\in H$, $G^+$ contains
	$(h,h^\varphi)^g = (h,h^\varphi)^{(y,1)\tau} = (h^\varphi, h^y)$ and hence we must have $h^y = (h^\varphi)^\varphi$ for all $h\in H$, that is to say, $\varphi^2=\iota_y$. Finally $G_\alpha = G^+_\alpha = \{(h,h^\varphi) :  \alpha^{(h, h^\varphi)} = (u^h)_0 = u_0\} \cong H_u$, and we know already that $G_\alpha$ is a 2-group. 
\end{proof}	

Now we apply the structure theorem from \cite{praeger2003finite} for biquasiprimitive groups. It turns out that only two of the various possible structures given in  Theorem 1.1 of \cite{praeger2003finite} can arise as groups of automorphisms of 4-valent oriented graphs of basic biquasiprimitive type. Note that
the stabiliser $G_\alpha = \{(h,h^\varphi)\mid h\in H_u\}\cong H_u$.  

\begin{Proposition}\label{twotypes}
	Under the assumptions of Proposition~\ref{iso}, the automorphism $\varphi$ is nontrivial, and $G$ has a unique minimal normal subgroup $N=\soc(G)$. Moreover $N=\Diag_\varphi(M\times M)\cong M$ where $M=\soc(H)\cong T^k$ for some simple group $T$ and $k\geq1$, and either
	\begin{enumerate}
		\item[(a)] $H$ is quasiprimitive and $M$ is its unique minimal normal subgroup, or
		\item[(b)] $H$ is not quasiprimitive and $M=R\times R^\varphi$ where $R, R^\varphi$ are intransitive  minimal normal subgroups of $H$. In this case $G^+$ has  two minimal normal subgroups, namely $K:=\Diag_\varphi(R\times R)$ and $L = \Diag_\varphi(R^\varphi\times R^\varphi)$, and these are the only minimal normal subgroups if $T$ is nonabelian. Moreover, $N=K\times L$ (so $k=2\ell$ and $K\cong L\cong R\cong T^\ell$). 
	\end{enumerate}
\end{Proposition}

\begin{proof}
	We examine the possibilities for the structure of $G$ given in \cite[Theorem 1.1]{praeger2003finite}. Since $G_\alpha$ is a 2-group, cases (b) and (c)(ii) do not arise, and since $G^+$ is faithful on $\Delta$, the possible cases are (a)(i) and (c)(i). Consider first case (a)(i). Since $|X|>4$, the element $g = (y,1)\tau$ does not centralise $G^+$. A straightforward computation shows that $C_{G^+}(g)$ consists of all pairs $(h,h^\varphi)$ such that $h\in C_H(\varphi)$. Thus $\varphi$ is nontrivial. Moreover in case (a)(i), $H$ is quasiprimitive on $V$ and the stabiliser $H_u\cong G_\alpha$ is a 2-group. 
	
	We now apply the O'nan-Scott Theorem for quasiprimitive groups from \cite{praeger1993nan}. This theorem tells us that if $H$ has more than one minimal normal subgroup then the stabilizer $H_u$ is not solvable. Thus $H$ has a unique minimal normal subgroup $M=\soc{H}\cong T^k$ where $T$ is a simple group and $k\geq 1$. Now $G^+$ has a minimal normal subgroup $N=\Diag_\varphi(M\times M)\cong M$ and since $G^+ \cong H$ it follows that $N$ is the unique minimal normal subgroup of $G^+$.
	
	It remains to consider case (c)(i). Here again, $G$ has a  unique minimal normal subgroup $N=\Diag_\varphi(M\times M)$, but in this case $M=\soc(H)=R\times R^\varphi$ where $R, R^\varphi$ are intransitive minimal normal subgroups of $H$. In particular $\varphi$ is nontrivial, and $R\cong R^\varphi\cong T^\ell$ and $M\cong T^k$ with $k=2\ell$. Here $K$ and $L$, as in part (b),  are the minimal normal subgroups of $G^+$, and are interchanged by $g$ (noting that, for $(h,h^\varphi)\in K$, the conjugate $(h,h^\varphi)^g = (h^\varphi, h^y)\in L$, since $\varphi^2=\iota_y$, and vice versa).

		If $T$ is nonabelian then since $R$ is a minimal normal subgroup of $H$, it follows that $H$ permutes the simple direct factors of $R$ (and $R^\varphi$) transitively. Hence these are the only minimal normal subgroups of $H$, and $K$ and $L$ are the only minimal normal subgroups of $G^+$. 
		
		On the other hand, if $T=C_p$ then as an $H_u$-module, $M$ has two composition factors each isomorphic to $R$. In particular, $H$ may have other minimal normal subgroups. However, for any such subgroup $S$ we have $S\cong R$ as there are just two composition factors and both are isomorphic to $R$. Also since $N$ is the unique minimal normal subgroup of $G$ it follows that $M = S \times S^\varphi$ also. 
\end{proof}

In summary if $(\Gamma, G)\in \OG(4)$ is basic and biquasiprimitive, then $N:= \soc(G)$ is the unique minimal normal subgroup of $G$, and is contained in $G^+$. In particular $N$ is transitive on the two $G^+$-orbits $\Delta$ and $\Delta^+$, and since $G_\alpha = G^+_\alpha$, it follows that $G^+ = NG_\alpha$.

Using the framework of Proposition~\ref{iso}, we can specify the neighbours of $\alpha = u_0$ and of $\alpha^{g^{-1}}=u_1$. 
We denote by $\Gamma_{out}(\gamma)$ and $\Gamma_{in}(\gamma)$ the 2-subsets of out-neighbours and in-neighbours of a vertex $\gamma$, respectively. Each of these two sets is an orbit of the stabiliser $G_\gamma$, and we can always choose an element of $G_\gamma$ that acts fixed-point-freely on $\Gamma(\gamma)$ (whether the induced group has order 2 or 4). For the vertex $\alpha$, such an element is of the form $(z^{\varphi^{-1}},z)$ for some $z\in (H_u)^\varphi$.  Since we did not specify above, let us now decide that the vertex $\beta=(u^y)_1$ in Proposition~\ref{iso} lies in $\Gamma_{in}(\alpha)$.

\begin{Lemma}\label{neighbours}
	Use the notation of Proposition~\ref{iso} (in particular that $g=(y,1)\tau$ and $\alpha=u_0$), and let $(z^{\varphi^{-1}},z)\in G_\alpha$ be  fixed-point-free on $\Gamma(\alpha)$,  for some $z\in (H_u)^\varphi$. Then
	\begin{enumerate}
		\item[(a)]  $\Gamma_{in}(\alpha)=\{  (u^y)_1, (u^{yz})_1 \}$ and $\Gamma_{out}(\alpha) = \{ u_1, (u^z)_1 \}$; and
		\item[(b)] for $\gamma :=\alpha^{g^{-1}} = u_1$,  $\Gamma_{in}(\gamma)=\{  u_0, (u^{yzy^{-1}})_0 \}$ and $\Gamma_{out}(\gamma) = \{ (u^{y^{-1}})_0, (u^{zy^{-1}})_0 \}$.
	\end{enumerate}
\end{Lemma}

\begin{proof}
	As mentioned above, we assume that  the vertex $\beta=\alpha^g=(u^y)_1$ in Proposition~\ref{iso} lies in $\Gamma_{in}(\alpha)$. As 
	$(z^{\varphi^{-1}},z)\in G_\alpha$ is  fixed-point-free on $\Gamma(\alpha)$, the second vertex in  $\Gamma_{in}(\alpha)$ is $\beta^{(z^{\varphi^{-1}},z)} = (u^{yz})_1$. Note that $g^{-1} = (1,y^{-1})\tau$. Applying $g^{-1}$ to $\{\alpha\}\cup\Gamma_{in}(\alpha)$ we find first that $\alpha^{g^{-1}}= u_1$ and then that $\Gamma_{in}(u_1)$ consists of the vertices $(u^y)_1^{g^{-1}} = u_0$ and $(u^{yz})_1^{g^{-1}} = (u^{yzy^{-1}})_1$. In particular $u_1\in\Gamma_{out}(u_0)$ and the second vertex in this set is therefore   $u_1^{(z^{\varphi^{-1}},z)} = (u^{z})_1$.
	This completes the proof of part (a). 
	Finally applying $g^{-1}$ to  $\{\alpha\}\cup\Gamma_{out}(\alpha)$ we find that  $\Gamma_{out}(u_1)$ consists of the vertices  $(u)_1^{g^{-1}} = (u^{y^{-1}})_0$ and $(u^{z})_1^{g^{-1}} = (u^{zy^{-1}})_0$.
\end{proof}
\section{Biquasiprimitive Basic Pairs: restricting the socle.}\label{SecRestricting}
We will now show that for any biquasiprimitive basic pair $(\Gamma, G) \in \OG(4)$, the unique minimal normal subgroup $N$ of $G$ is a direct product of $k$ finite simple groups where $k$ takes one of only several possible values depending on the structure of $G$. 
We deduce these values of $k$ by separately considering the cases when $N$ is abelian and nonabelian.

We first consider the case where the minimal normal subgroup $N = \soc(G)$ is abelian. Since $N$ is contained in $G^+$, this implies that $N$ acts transitively and hence regularly on $\Delta$ (and $\Delta')$. In particular, $\Gamma$ is a bi-Cayley graph over $N$, that is, $\Gamma = \BiCay(N, \emptyset, \emptyset, S)$, and  $N = C_p^k$ for some $k \geq 1$.

\begin{Lemma}\label{abelianBound}
	Let $(\Gamma, G) \in \OG(4)$ be basic of biquasiprimitive type and suppose that $N = \soc(G)$ is abelian. Then $N = C_p^k$ with $k \leq 2$ and $p$ an odd prime.
\end{Lemma}
\begin{proof}
	Since $N = C_p^k$ is an abelian  normal subgroup of $G$ contained in $G^+$, $N$ is regular on the two $G^+$-orbits $\Delta$ and $\Delta'$, and $\Gamma \cong\BiCay(N, \emptyset, \emptyset, S)$, for some subset $S \subseteq N$ as defined in Subsection \ref{ssBiCay}.
	We can view $\Delta$ and $\Delta'$ as two copies of the group $N$, so $\Delta = N_0$ and $\Delta' = N_1$, with each vertex $n_0 \in \Delta$ adjacent to $(n + s)_1 \in \Delta'$, where $s \in S$. Since $\Gamma$ is connected and 4-valent, it follows that $\langle S \rangle = N$ and $|S| = 4$, in particular $k \leq 4$.  
	
	Suppose first that $k = 4$ and $S = \{s_1, s_2, s_3, s_4\} \subseteq N$. We may view $N = C_p^4$ as a 4-dimensional vector space over the finite field $\mathbb{F}_p.$ Since $S$ generates $N$, it follows that the elements of $S$ viewed as vectors of this  space are linearly independent. 
	
	Since $\Gamma$ is connected, there is a path from $(0_N)_0 \in \Delta$ to $(0_N)_1 \in \Delta'$. Moreover, this path has odd length since $\Delta$ and $\Delta'$ are independent sets. 
	Let $P$ be a path from $(0_N)_0$  to $(0_N)_1$, then $P$ is of the form 
	$$(0_N)_0, (t_1)_1, (t_1-t_2)_0, (t_1-t_2+t_3)_1,...,(t_1-t_2+t_3 - ...+t_r)_1 = (0_N)_1$$
	where each $t_i \in S$ for $1\leq i\leq r$. In particular, since $P$ has odd length, $r$ is an odd integer.
	
	Now consider the expression $t_1-t_2+t_3 - ...+t_r = 0_N$. For each $j$ with $1\leq j\leq 4$, let $\alpha_j$ be the number of odd $i$ such that $t_i = s_j$, and let $\beta_j$ be the number of even $i$ such that $t_i = s_j$. Notice that since the length of the path $P$ is odd, the sum $\sum_{j=1}^4 \alpha_j$ is equal to exactly $1 + \sum_{j=1}^4 \beta_j$, and so $ \sum_{j=1}^{4} (\alpha_j-\beta_j) =1 $ (an equation over the integers $\mathbb{Z}$).
	
	On the other hand since $t_1-t_2+t_3 - ...+t_r = 0_N$ (an equation in the group $N$), we get that 
	$$0_N = \sum_{j=1}^{4} (\alpha_j-\beta_j)s_j,$$
	and since the elements $s_j$ of $S$ are linearly independent, we know that  $\alpha_j \equiv \beta_j$ mod $p$, for each $j$. Hence
	$$0 \equiv \sum_{j=1}^{4} (\alpha_j-\beta_j)\mod p,$$ contradicting the fact that $ \sum_{j=1}^{4} (\alpha_j-\beta_j) =1$, thus $k\neq 4$.
	
	Next suppose that $N = C_p^k$ with $k= 3$. Since $k$ is odd, it follows by Proposition \ref{twotypes} that $G^+ = NG_\alpha$ is quasiprimitive on $\Delta = N$. In particular since $N$ is regular on $\Delta$, no proper non-trivial subgroup of $N$ is normal in $G^+$. Since $N$ acts trivially on itself by conjugation, this implies that conjugation by $G_\alpha$ fixes no proper non-trivial subgroup of $N$.
	
	However, $G_\alpha$ is a 2-group, and $N$ has exactly $p^2 + p +1$ subgroups of order $p$. Since this number is odd, some subgroup must be left fixed under conjugation by $G_\alpha$ and hence must be normal in $G^+$, a contradiction. Therefore $k \leq 2$.
	
	To see that $p$ must be odd notice that if $k = 2$ then again conjugation by $G_\alpha$ cannot fix any of the $p+1$  subgroups of $N$ of order $p$ implying that $p \neq 2$. On the other hand, if $k = 1$ then the fact that $|V\Gamma|>4$ implies that $N \neq C_2$.
\end{proof}

The next lemma concerns the case when $N =\soc(G)$ is nonabelian. The proof develops ideas used to prove a similar result for quasiprimitive basic pairs in \cite[Lemma 6.2]{al2015finite}.

\begin{Lemma}\label{nonabelianBound}
	Let $(\Gamma, G) \in \OG(4)$ be basic of biquasiprimitive type and suppose that $N = \soc(G)$ is nonabelian.  Then  either
	\begin{enumerate}[(a)]
		\item  $N$ is a minimal normal subgroup of $G^+$ and $N = T^k$, for some nonabelian simple group $T$ and $k \in \{1,2,4\}$; or
		\item  $N = K \times K^g$ where $g\in G \backslash G^+$, and $K = T^\ell$ is a minimal normal subgroup of $G^+$ with $T$ a nonabelian simple group and $\ell \in \{1,2,4\}$. In particular, $N\cong T^k$ with $k = 2 \ell$.
	\end{enumerate}
\end{Lemma}
\begin{proof}
	
	Let $(\Gamma, G) \in \OG(4)$ be as in the statement of the theorem and suppose that $N= \soc(G)$ is nonabelian. The possible cases (a) and (b) here correspond directly to the two cases of Proposition \ref{twotypes}. The group $K$ in case (b) is the subgroup 	$K := \{(r, r^\varphi)  :  r\in R\}$ of Proposition \ref{twotypes},
 and so $K^g = \{(r^\varphi, r^y)  :  r\in R\}$, where $R$ is an intransitive minimal normal subgroup of $H$.
	
	Since $N = \soc(G)$ is nonabelian, it follows that $N$ is a direct product of isomorphic nonabelian simple groups $T$. In particular, $N = T^k$ for $k\geq1$, and in case (b), $k = 2\ell$ where $K = T^\ell$ and $\ell \geq 1$.
	We will now show that $k$ divides $ 4$ in case (a) and $\ell$ divides  $ 4$ in case (b). As $N = \soc(G)$, we will identify $N$ with its group of inner automorphisms Inn($N$), and regard $G$ as a subgroup of $\Aut(N) \cong \Aut(T) \wr \Sym(k)$. The representations of elements will therefore be different from Proposition \ref{twotypes}.
	
	Let $s\geq 1$ be the largest integer such that $G$ acts transitively on the oriented $s$-arcs of $\Gamma$. By Lemma \ref{regular_s_arcs}, this implies that $G$ is regular on the oriented $s$-arcs of $\Gamma$. Consider now an  oriented $s$-arc $(v_0,v_1,...,v_s)$ of $\Gamma$, and suppose that the pointwise stabilizer $ G_{{v_0},...,{v_{s-1}}} $ of order 2 is generated by the element $h_1$, that is, $G_{{v_0},...,{v_{s-1}}} = \langle h_1 \rangle \cong C_2$.
	
	Now let $g \in G\backslash G^+$ be an automorphism of $\Gamma$ taking the oriented $s$-arc $(v_0,v_1,...,v_s)$ to the oriented $s$-arc $(v_1, v_2,...,v_s, v_{s+1})$ where $v_{s+1}$ is some out-neighbour of $v_s$. For each $2 \leq i \leq s$, define $h_i := h^{g^{-1}}_{i-1}$. It is clear that for each $i \leq s$ we have $$ G_{{v_0},...,{v_{s-i}}} = \langle h_1, ..., h_i\rangle.$$
	
	We may write the automorphisms $h_1, g \in G$ as elements of $\Aut(N) \cong \Aut(T) \wr \Sym(k)$, so that $h_1 = f\sigma$ and $g = f'\tau$ where $f, f' \in \Aut(T)^k$ and $\sigma, \tau \in \Sym(k)$. In fact in case(b), $\sigma, \tau \in \Sym(\ell) \wr \Sym(2)$ with $\sigma \in \Sym(\ell)\times \Sym(\ell)$. In either case, $h_1^2 =1$ implies that $\sigma^2 =1$.
	
	Now let $\pi$ denote the projection map $\pi: \Aut(N) \rightarrow \Sym(k)$
	, so that $(h_1)\pi = \sigma$ and $(g)\pi = \tau$, and let $P := (G^+)\pi = (NG_{v_{0}})\pi = (G_{v_0})\pi$. Note that $P$ is a 2-group since $G_{v_0}$ is a 2-group, and moreover  $$P = (G_{v_0})\pi = \langle h_1,h_2,..., h_s\rangle\pi = \langle \sigma, \sigma^{\tau^{-1}},..., \sigma^{\tau^{-(s-1)}}\rangle.$$
	
	We claim that $\sigma$ is not contained in any proper $\tau$-invariant subgroup of $P$. Suppose on the contrary that $\bar{P}$ is a proper $\tau$-invariant subgroup of $P$ containing $\sigma$. Since $\bar{P}$ is $\tau$-invariant it follows that $\sigma^{\tau^{-i}} \in \bar{P}$ for all $i \in \mathbb{Z}$, implying that $P\leq \bar{P}$ and hence that $P = \bar{P}$, a contradiction.
	
	Notice that $P$ is a subgroup of index 1 or 2 of $(G)\pi$, and $P$ is transitive in case (a) or  has two orbits of length $\ell$ in case (b), so $k$ divides $|P|$, or $\ell$ divides $|P|$ respectively. We will now consider separately the two possibilities for the index of $P$ in $(G)\pi$ and show that in either case $|P|$ divides 4.
	
	Suppose first that $P = (G)\pi$ and let $M$ be a maximal subgroup of $P$ containing $\langle \sigma \rangle$. Since $P$ is a 2-group it follows that $M$ is normal in $P$ and, in particular must be $\tau$-invariant. Since $\sigma$ cannot be contained in any proper $\tau$-invariant subgroup of $P$, it follows that $P = \langle \sigma \rangle$ with order at most 2, and therefore that $k \leq 2$ (or $\ell \leq 2$).
	
	Suppose on the other hand that $P$ is an index 2 subgroup of $(G)\pi$, in particular, this implies that the order of $\sigma$ is 2. In this case $\tau\in (G)\pi \backslash P$. However $g^2 \in G^+$ and hence $\tau^2 \in P$. Furthermore, $\sigma$ does not lie in any proper $\tau$-invariant subgroup $H$ of $P$ (otherwise we can use the same argument as in the previous paragraph to show that $H = P$, a contradiction).
	
	Now let $L := \Phi(P)$, the Frattini subgroup of $P$ and note that $P/L$ is elementary abelian. Then $L$ is $\tau$-invariant since $\tau$ normalizes $P$, so $L$ does not contain $\sigma$. Setting $J := \langle L, \sigma \rangle$, it follows that $J/L$ has order 2, and conjugation  by $\tau^{-1}$ maps $J/L$ to $(J^{\tau^{-1}})/L$. However, $J$ is normal in $P$ since $P/L$ is elementary abelian. In particular, since $\tau^2 \in P$, conjugation by $\tau^2$ fixes $J$ and $J/L$. 
	
	Therefore repeated applications of conjugation by $\tau$ simply interchange the two (possibly equal) subgroups $J/L$ and $(J^{\tau^{-1}})/L$ of $P/L$ and each generator $\sigma^{\tau^{-i}}$ of $P$,  lies in either $J$ or $J^{\tau^{-1}}.$ It follows that $P/L$ is generated by $J/L$ and $J^{\tau^{-1}}/L$, and hence that $P/L \cong C_2^c$ for $c \leq 2$. 
	
	If $c =1$ then $P \cong \langle \sigma \rangle$ and this implies that $k= 2$ in case (a), or  that $\ell =2$ in case (b).
	On the other hand, if $P/L \cong C_2^2$, then $P = \langle \sigma, \sigma^{\tau^{-1}} \rangle = \langle h_1, h_2 \rangle \pi$ and since we know that $\langle h_1, h_2 \rangle =  G_{{v_0},...,{v_{s-2}}}$ has order $2^2 = 4$, it follows that the order of $P$ divides 4. In particular $k$ divides 4 in case (a), or $\ell$ divides 4 in case (b). This completes the proof.
\end{proof}

The first assertions of Theorem \ref{MainResult} now follow directly from Proposition \ref{twotypes} together with Lemmas \ref{abelianBound} and \ref{nonabelianBound}. 

\section{Constructing Biquasiprimitve Pairs}\label{SecConstructions}
 In this section we complete the proof of Theorem \ref{MainResult}. We do this by explicitly constructing examples of biquasiprimitive pairs corresponding to the different cases of Theorem \ref{MainResult}. In each of the three cases (a) - (c) of Theorem \ref{MainResult}, the parameter $k$ (the number of simple direct factors of the socle of $G$) can take several different values. In case (a) there are two possibilities for the value of $k$, while in each of the cases (b) and (c) there are three possibilities.
 
 Thus Theorem \ref{MainResult} gives a total of eight different possibilities for the structure of $\soc(G)$ of a biquasiprimitive pair $(\Gamma, G)$ where the number of simple direct factors is taken into account. To complete the proof, we therefore provide eight infinite families of biquasiprimitive basic pairs corresponding to these distinct cases. 

In Subsection \ref{Methods} we will describe two methods for constructing biquasiprimitive basic pairs. In short, Method \ref{BiMethod} uses the standard bi-Cayley graph construction described in Subsection \ref{ssBiCay}, while Method \ref{CosetMethod} is a more general coset graph construction developed from Proposition \ref{iso}. All of our  constructions of biquasiprimitive pairs will use one of these two methods.

The examples constructed to complete the proof of Theorem \ref{MainResult} are given in Constructions \ref{Abelian k=1} - \ref{Nonabelian ell=4} of this section. Table \ref{ConstructionTable} shows all of these constructions along with the explicit simple group $T$ used in each case. The `Methods Used' column refers to one of the two methods developed in Subsection \ref{Methods} for producing biquasiprimitive pairs. The construction numbers are included for easy reference.

	\begin{table}[h]
	
	\begin{tabular}{ l l l l l}
		\hline
		Case described in Theorem \ref{MainResult} & Value of $k$ & Simple Group $T$ & Construction \# & Method Used \\ 
		\hline
		Case (a) & $k = 1$ &$\mathbb{Z}_p$, $p \equiv 1$ mod 4 & Construction \ref{Abelian k=1} & Method \ref{BiMethod}\\  
	& $k = 2$ & $\mathbb{Z}_p$, $p \equiv 3$ mod 4  & Construction \ref{Abelian k=2}  & Method \ref{BiMethod} \\
		Case (b) & $k = 1$ & Alt($n$), $n\geq5$, odd & Construction \ref{Nonabelian k=1} & Method \ref{BiMethod}\\
	 & $k = 2$ & Alt($n$), $n\geq5$, odd & Construction \ref{Nonabelian k=2} & Method \ref{BiMethod}\\
		 & $k = 4$ & $\PSL(2,p) ,$ $ p\geq 7 $& Construction \ref{Nonabelian k=4} & Method \ref{CosetMethod}\\
		 	Case (c) & $k = 2$ & Alt($n$), $n\geq5$, odd & Construction \ref{Nonabelian ell=1} & Method \ref{BiMethod} \\
		 & $k = 4$ & $\PSL(2,p) ,$ $ p\geq 7 $ & Construction \ref{Nonabelian ell=2} & Method \ref{CosetMethod}\\
		 & $k = 8$ & $\PSL(2,p) ,$ $ p\geq 7 $ & Construction \ref{Nonabelian ell=4} & Method \ref{CosetMethod}\\
		\hline
	\end{tabular}\caption{Constructions of basic biquasiprimitive pairs $(\Gamma, G)$ with $\soc(G) \cong T^k$ as described in the various cases of Theorem \ref{MainResult}. }\label{ConstructionTable}
\end{table}

\subsection{Two Methods for Constructing Biquasiprimitive Pairs}\label{Methods}
One way to construct biquasiprimitive pairs is using the `standard' bi-Cayley construction described in Subsection \ref{ssBiCay}. Specifically, if $(\Gamma, G)\in \OG(4)$ is basic and biquasiprimitive, and the unique minimal normal subgroup $N$ of $G$ is semiregular on the two $G^+$-orbits, then we can take $\Gamma$ to be a bi-Cayley graph $\Gamma:= \BiCay(N, \emptyset, \emptyset, S)$ (for some subset $S$ of $N$ of cardinality 4).

In our constructions involving bi-Cayley graphs presented in the form $\Gamma = \BiCay(N, \emptyset, \emptyset, S)$ we will always use the natural labelling of the vertex $V\Gamma$. That is, we let $V\Gamma = N_0\cup N_1$ consisting of two copies of the group $N$ with each vertex labelled $(n)_\epsilon$ for $n\in N$ and $\epsilon \in \{0,1\}$. 

Suppose now that $\Gamma = \BiCay(N, \emptyset, \emptyset, S)$ where $S = S^{-1}$. Of course, such a graph is bipartite with $N_0$ and $N_1$ forming the bipartition. In order to show that a $\Gamma$ is connected, it suffices to show that the vertex set $N_0$ lies in a single connected component of $\Gamma$, or in other words that there is a path from $(1_N)_0$ to $(n)_0$ for any $n\in N$ (vertex-transitivity then ensures that this holds for $N_1$ also). Any such path must have even length and consist of repeated left multiplication in $N$ by an element of $S$ followed by an element of $S^{-1} = S$. In particular, the graph $\Gamma$ is connected if $\langle S^2\rangle = N$. 
 
Hence we have the following simple method for constructing biquasiprimitive basic pairs $(\Gamma, G)$.

\begin{method}\label{BiMethod} Take a group $N = T^k$ where $T$ is a simple group and $k \geq 1$, and construct a pair $(\Gamma, G)$ with $N = \soc(G)$ as follows:
	\begin{enumerate}
 	\item Let $\Gamma = (N, \emptyset, \emptyset , S)$, where $S\subset N$ such that $S= S^{-1}$, $|S| = 4$,  and $\langle S \rangle = \langle S^2 \rangle = N$.
	\item  Take a group $G$ with $N \leq G \leq N_{\Aut(\Gamma)}(N)$  for which $\Gamma$ is $G$-oriented. This gives $(\Gamma, G)\in \OG(4)$.
	\item Show that $N$ is the unique minimal normal subgroup of $G$ to get that $(\Gamma, G)$ is biquasiprimitive. 
\end{enumerate}
\end{method}
Note that  $N_{\Aut(\Gamma)}(N)$ (the normalizer of $N$ in $\Aut(\Gamma)$) was determined in \cite[Theorem 1.1]{zhou2016automorphisms}. In fact, in our constructions we will only use the following fact which follows from \cite[Lemmas 3.2 and 3.3]{zhou2016automorphisms}.

\begin{Proposition}\label{AutBiCay}
	Let $\Gamma = \BiCay(N,\emptyset,\emptyset, S)$ as defined in Subsection \ref{ssBiCay} with $S = S^{-1}$. Suppose $\alpha \in \Aut(N)$ with $S^\alpha = S$. 
	Then the permutations $\delta_\alpha$ and $\sigma_\alpha$ of $V\Gamma$ where $\delta_\alpha: $ $x_\varepsilon \mapsto (x^\alpha)_{1-\varepsilon}$, and $\sigma_\alpha:$ $x \mapsto (x^\alpha)_{\varepsilon}$ for $x\in N$ and $\varepsilon \in \{0,1\}$ are both automorphisms of $\Gamma$. Moreover both $\delta_\alpha$ and $\sigma_\alpha$ normalize the semi-regular subgroup $N\leq \Aut(\Gamma)$.
	\end{Proposition}

More generally, we may construct biquasiprimitive pairs $(\Gamma, G)$ by using the coset graph construction. For a group $G$, a proper subgroup $S$, and an element $g \in G$, the \textit{coset graph} $\Gamma = \Cos(G, S, g)$ is the undirected graph with vertex set $\{Sx : x \in G\}$ and edges $\{Sx, Sy\}$ if and only if $xy^{-1}$ or $yx^{-1} \in SgS$. The group $G$ acting by right multiplication on $V\Gamma$ induces a vertex-transitive and edge-transitive group of automorphisms of $\Gamma$, and this action is faithful if and only if $S$ is core-free in $G$.
Furthermore, the graph $\Gamma$ is connected if and only if $\langle S, g\rangle = G$, and is $G$-oriented and 4-valent if and only if $g^{-1} \notin SgS$ and $|S: S\cap S^g| = 2$ (see discussion at the beginning of \cite[Section 5]{al2015finite}). In summary, if $\Gamma = \Cos(G, S, g)$, then $(\Gamma, G) \in \OG(4)$ if and only if 
\begin{enumerate}[(1)]
	\item \hspace{1cm} $S$ is core-free in $G$, \hspace{.5cm} $g^{-1} \notin SgS$, \hspace{.5cm} $|S: S\cap S^g| = 2$, \hspace{.2cm} and \hspace{.2cm} $\langle S, g\rangle = G$.
\end{enumerate}
Moreover, for each pair $(\Gamma, G) \in \OG(4)$ there exist $S \le G$ and $g \in G$ such that $\Gamma = \Cos(G, S, g)$ and (1) holds. 

We can use Proposition \ref{iso} on the structure of biquasiprimitive basic pairs $(\Gamma, G) \in \OG(4)$ together with the coset graph construction given above to find examples of coset graphs of biquasiprimitive type. We begin by providing a general construction which uses a permutation group $H$ (with some prescribed properties) to produce a pair $(\Gamma, G)$ where $\Gamma$ is a coset graph for $G$, and $G$ has an index 2 subgroup isomorphic to $H$. In the remainder of this section we will show that under certain conditions the pairs $(\Gamma, G)$ constructed in this way are biquasiprimitive.

\begin{Construction}\label{GeneralCoset}
	Take a permutation group $H$, a proper subgroup $V\leq H$, a non-identity element $y\in H$, and an automorphism  $\varphi \in \Aut(H)$ such that $\varphi^2 = \iota_y$.
	
	Now consider the group $H \wr S_2$ and define two of its subgroups $G^+ := \Diag_{\varphi}(H\times H)$, and  $S := \Diag_{\varphi}(V \times V)$. Also define an element $g := (y,1)(12)\in  H \wr S_2$. Finally construct the graph-group pair $(\Gamma, G)$ where $G := \langle G^+, g \rangle \leq H \wr S_2$ and $\Gamma := \Cos(G, S, g)$.
	
\end{Construction}
It is clear that the construction of the group $G$ in this way corresponds to the formulation of the biquasiprimitive permutation group $G$ given in Proposition \ref{iso}. Notice in particular that using this construction, the pair $(\Gamma, G)$ is completely determined by the choices of appropriate $H, V, y$ and $\varphi$. Hence we will say that a tuple $(H, V, y, \varphi)$  is \textit{appropriate}  if $H, V, y$ and $\varphi$ satisfy the conditions of Construction \ref{GeneralCoset}. In many of the constructions that follow, we will simply apply Construction \ref{GeneralCoset} on an appropriate $(H, V, y, \varphi)$ to create pairs $(\Gamma, G)$.
The following lemma gives a sufficient condition for $(\Gamma, G)$ constructed in this way to be a member of $\OG(4)$.

\begin{Lemma}\label{cond2}
	Let  ($\Gamma, G$) be a graph-group pair constructed using Construction \ref{GeneralCoset} on  an appropriate $(H, V, y, \varphi)$. Then $(\Gamma, G) \in \OG(4)$ if 
	\begin{enumerate}[(2)]
		\item \hspace{1cm} $V$ is core-free in $H$, \hspace{.5cm} $y \notin VV^\varphi$, \hspace{.5cm} $|V: V\cap V^\varphi| = 2$, \hspace{.2cm} and \hspace{.2cm} $\langle V, y\rangle = H$.
	\end{enumerate}	
\end{Lemma}

\begin{proof}
	Let $G^+$ and $S$ be the subgroups of $G$ defined in the construction, and let $\Gamma = \Cos(G, S, g)$. Suppose that (2) holds. We will show that $(\Gamma, G) \in \OG(4)$ by showing that (1) holds also.  
	
	First, since $H \cong G^+$, $S \cong V$, and $V$ is core-free in $H$, it follows $S$ is core-free in $G^+$ and hence is core-free in $G$.
	Next, we will show that $y\notin V V^\varphi$ implies that $g^{-1} \notin SgS$. 
	Notice that $g^{-1} = (1,y^{-1})(12)$, while for any element $z \in SgS$, $z = (s,s^\varphi)(y,1)(12)(t,t^\varphi) = (syt^\varphi, s^\varphi t)(12)$ for some $s,t \in V$. Thus if $g^{-1} = z$ for some $z \in SgS$, then $1 = syt^\varphi$ and hence $y \in VV^\varphi$.
	
	For the last two conditions notice that if we take $x \in G^+$ then $x^g = (h, h^\varphi)^g = (h^\varphi, h^y)$ for some $h\in H$. In particular, for $s \in S$ we have $s^g = (t^\varphi, t^y)$ where $t \in V$. So $s^g \in S$ if and only if $t^\varphi \in V$. Since $V \cong S$ we get that $|S:S\cap S^g| = |V: V \cap V^\varphi|$. 
	
	Finally, it is easy to check that $g^2 = (y,y) \in G^+$. Hence if $\langle V, y \rangle = H$, then $\langle S, g^2 \rangle = \Diag_{\varphi}(\langle V, y\rangle\times \langle V, y\rangle) = \Diag_{\varphi}(H\times H) = G^+$ and so $\langle S, g \rangle = G$.
\end{proof}

Hence we have an easy condition for ensuring that pairs $(\Gamma, G)$ formed using Construction \ref{GeneralCoset} are contained in $\OG(4)$. Our next goal is to provide a simple condition under which such pairs are biquasiprimitive.

\begin{Lemma}\label{Cosmin}
		Let  ($\Gamma, G$) be a graph-group pair constructed using Construction \ref{GeneralCoset} on an appropriate $(H, V, y, \varphi)$. Let $G^+$ and $S$ as defined in that construction. 
		Then
		\begin{itemize}
			\item Every minimal normal subgroup of $G$ is contained in $G^+$.
			\item If  $\soc(G^+) \cong \soc(H)$ is a minimal normal subgroup of $G$ then it is the unique minimal normal subgroup of $G$.
		\end{itemize}

\end{Lemma}
\begin{proof}
	For the first part, notice that $|G:G^+| = 2$ (since $G = \langle G^+, g\rangle$, $g$ noramlizes $G^+$ and $g^2 = (y,y) \in G^+$). Now consider a minimal normal subgroup $N$ of $G$ and suppose that $G^+\cap N \neq N$. By the minimality of $N$ it follows that $G^+ \cap N = 1$ implying that $G = G^+ \times N$. But this implies that $N = \langle g \rangle$ with order 2, a contradiction since $g^2 = (y,y) \neq 1$. Hence $N\leq G^+$.
	
	For the second part, suppose that $\soc(G^+)$ is a minimal normal subgroup of $G$ and take a minimal normal subgroup $N$ of $G$ with $N \neq \soc(G^+)$. Then by the first part, $N$ is normal in $G^+$. In particular, $N\cap \soc(G^+) \neq 1$, a contradiction.
\end{proof}

The above result gives the following corollary.

\begin{Corollary}\label{CorMin}
	Suppose that $(\Gamma, G) \in \OG(4)$ where $(\Gamma, G)$ is constructed by Construction \ref{GeneralCoset} on an appropriate $(H, V, y, \varphi)$. Let $G^+$, and $S$ as defined in that construction. 
	
	Suppose further that $H = MV$ where $M = \soc(H) \cong T^k$ for some simple group $T$ and $k\geq1$. If  $\soc(G^+) \cong \soc(H)$ is a minimal normal subgroup of $G$, then $(\Gamma, G)$ is biquasiprimitive.
\end{Corollary}
\begin{proof}
		The vertex set of $\Gamma$ is the set of right cosets of $S$ in $G$. Hence there are two $G^+$-orbits, namely $\Delta = \{Sx:x\in G^+\}$ and $\Delta' = \{Sgx:x\in G^+\}$. If $N = \soc(G^+) \cong M$ is a minimal normal subgroup of $G$ then $N$ is the unique such subgroup by Lemma \ref{Cosmin}. Moreover, the condition $H = MV$ implies that $G^+ \cong NS$ so  $N$ is transitive on the two $G^+$-orbits $\Delta$ and $\Delta'$, and hence $G$ is biquasiprimitive on $V\Gamma$.
\end{proof}

The above results now provide the following method for constructing biquasiprimitive pairs in $\OG(4)$. 
\begin{method}\label{CosetMethod}
 Take a group $M = T^k$ for some simple group $T$ and $k\geq1$, and define a group $H := MV$ where $M = \soc(H)$ and $V$ is a proper subgroup $V\leq H$. Also take a non-identity $y \in H$ and an automorphism $\varphi \in \Aut(H)$ with $\varphi^2 = \iota_y$, so $(H, V, y ,\varphi)$ is appropriate.
 \begin{enumerate}
		\item Apply Construction \ref{GeneralCoset} on $(H, V, y$, $\varphi$) to create a pair $(\Gamma, G)$.
		\item Show that $H, V, y$ and $\varphi$ satisfy condition (2) of Lemma \ref{cond2} to get that $(\Gamma, G) \in \OG(4)$.
		\item Show that $\soc(G^+) \cong M$ is a minimal normal subgroup of $G$ to get that $(\Gamma, G)$ is biquasiprimitive (by Corollary \ref{CorMin}).
	\end{enumerate}
\end{method}

\subsection{Constructing Examples}
We now provide constructions of basic biquasiprimitve pairs $(\Gamma, G) \in \OG(4)$ with the various possible structures for $\soc(G)$ as described in cases (a) - (c) of Theorem \ref{MainResult}. We will use both the bi-Cayley graph construction described in Subsection \ref{ssBiCay} (Method \ref{BiMethod}) and the coset graph construction developed in the last part of the previous section (Method \ref{CosetMethod}).

We begin with examples of biquasiprimitive basic pairs ($\Gamma, G)$ with $\soc(G)$ abelian. Note that all 4-valent bi-Cayley graphs over an abelian group are arc-transitive \cite[Proposition 1.3]{conder2016edge}.

\begin{Construction}\label{Abelian k=1}
	Take a prime $p \equiv 1$ mod 4 and let  $q \in \mathbb{Z}_p$ such that $q^2 \equiv -1$ mod $p$.
	Let $\Gamma = \BiCay(N, \emptyset, \emptyset, S)$ with vertex set $N_0 \cup N_1$, where $N = \mathbb{Z}_p$ and $S = \{\pm 1, \pm q\}$. 
	Define a permutation $\delta$ of the vertices of $\Gamma$ by $x^\delta_\varepsilon = (x\cdot q)_{1-\varepsilon}$ for $\varepsilon \in \{0,1\}$, and  set $G := N \rtimes\langle \delta\rangle$. 
\end{Construction}

\begin{Lemma}\label{firstConstructionLemma}
	For $\Gamma ,G$ as in Construction \ref{Abelian k=1}, $(\Gamma, G) \in \OG(4)$ and is basic of biquasiprimitive type with $\soc(G)$ as described in Theorem \ref{MainResult} case (a) with $k = 1$. 
\end{Lemma}

\begin{proof}
	Since $|S| = 4$ and $\langle S \rangle = \langle S^2\rangle = N$  it follows that $\Gamma$ is 4-valent and connected. Also by Proposition \ref{AutBiCay}, $\delta \in \Aut(\Gamma)$ since it is induced by an automorphism of $N$ fixing $S$ setwise. Notice that the automorphism $\delta$ has order 4 and that the stabilizer of the vertex $(0)_0$ is $\langle \delta^2 \rangle \cong C_2$. This group has two orbits of length two on the neighbourhood of $(0)_0$, namely $\{(1)_1, (-1)_1\}$ and $\{(q)_1, (-q)_1\}$. 
	
	Now, any automorphism $g\in G$ is of the form $g = n \delta^i$ with $n\in N$ and $i \in \{1..4\}$. In particular, any automorphism taking the vertex $(0)_0$ to its neighbour $(1)_1$ must be of the form $g = n \delta^i$ with $n\in N$ and $i \in \{1,3\}$. This gives just two possibilities for such an automorphism namely $g_1 = q^3 \delta$ and $g_2 = q \delta^3$ where $q\in N$.
	These two automorphisms map $(1)_1$ to  $(1+q)_0$ and $(1-q)_0$ respectively. Thus no element of $G$ can reverse edges and $\Gamma$ is $G$-oriented. 
	
	Since the only proper non-trivial normal subgroups of $G$ are $N$ and $N\langle \delta^2 \rangle$ it follows that $(\Gamma, G)$ is basic of biquasiprimitve type.
\end{proof}

\begin{Construction} \label{Abelian k=2}
	Let $\Gamma = \BiCay(N, \emptyset, \emptyset, S)$ where $N = \mathbb{Z}_p^2$ for a prime $p \equiv 3$ mod 4, and $S = \{\pm(1,0), \pm(0,1)\}$. 
	Let $\delta$ be a permutation of $V\Gamma$  taking a vertex $(x,y)_{\epsilon}$ to $(y,-x)_{1-\epsilon}$ where $x, y\in \mathbb{Z}_p$ and $\varepsilon\in\{0,1\}$, and let $G := N\rtimes \langle \delta \rangle$.
\end{Construction} 

\begin{Lemma}
	For $\Gamma ,G$ as in Construction \ref{Abelian k=2}, $(\Gamma, G) \in \OG(4)$ and is basic of biquasiprimitive type with $\soc(G)$ as described in Theorem \ref{MainResult} case (a) with $k = 2$. 
	\end{Lemma}

\begin{proof}
	First note that $|S| = 4$ and $\langle S \rangle = \langle S^2\rangle = N$  so $\Gamma$ is 4-valent and connected. Also by Proposition \ref{AutBiCay}, $\delta \in \Aut(\Gamma)$. Furthermore, the automorphism  $\delta$ has order $4$, and for the vertex $\alpha = (0,0)_0$, we have $G_\alpha = \langle\delta^2\rangle \cong C_2$ with two orbits of length two on the neighbourhood of $\alpha$. Moreover any automorphism in $G$ taking the vertex $(0,0)_0$ to its neighbour $(1,0)_1$ must be either $g_1 = (0,1)\delta$ or $g_2 = (0,-1)\delta^3$ where $(0,1)$ and $(0,-1)$ are elements of $N$. However, neither of these automorphisms maps $(1,0)_1$  to $(0,0)_0$ and so no $g\in G$ can reverse edges of $\Gamma$ and $(\Gamma, G) \in \OG(4)$.
	
	To show that $(\Gamma, G)$ is basic of biquasiprimitive type, notice that the setwise stabilizer $G^+$ of the two parts $N_0$ and $N_1$ of $V\Gamma$ is $N \rtimes\langle \delta^2 \rangle$, with $\delta^2$ acting as inversion  on $N$. Hence the nontrivial normal subgroups of $G^+$ are $N$, and the subgroups of $N$ isomorphic to $\mathbb{Z}_p$ (all intransitive on $N_0$ since $N$ is regular). Therefore we need to check that none of the subgroups of $N$ of order $p$ is normal in $G$. 
	
	To this end, notice that the subgroups corresponding to the direct factors of $N$ are swapped by conjugation by $\delta$ in $G$, and hence aren't normal. All other nontrivial proper subgroups of $N$ are of the form  $\langle(1, x)\rangle$ with $x \in \mathbb{Z}^*_p$. Hence if $\langle(1,x)\rangle^\delta = \langle(x,-1)\rangle = \langle(1,x)\rangle,$ then $c(1,x) = (x,-1)$ for some $c \in \mathbb{Z}^*_p$. It follows that $c = x$ and so $x^2 \equiv -1 $ mod $p$, but this is impossible since $p \equiv 3$ mod 4. Thus the only proper non-trivial normal subgroups of $G$ are $N$ and $G^+$, both of which are transitive on the two biparts of $V\Gamma$.
\end{proof}

Next we give constructions of biquasiprimitive basic pairs $(\Gamma, G)\in \OG(4)$ with $\soc(G)$ nonabelian. Note that any nonabelian simple group $T$ can be generated by an involution and an element of prime order \cite{king2017generation}. In particular all nonabelian simple groups can be generated by two elements. In each of our constructions of biquasiprimitive pairs with nonabelian socle we will use a simple group $T$ and a generating pair $\{a,b\}$ with prescribed properties. 

We begin with constructions of biquasiprimitive basic pairs $(\Gamma, G)\in \OG(4)$ with $\soc(G)$ nonabelian and as described in Theorem \ref{MainResult} case (b). 

\begin{Construction}\label {Nonabelian k=1} Let $T$ be a nonabelian simple group, and let $\{a,b\}$ be a generating set for $T$ where $a$ is an involution and the elements $b$ and $ab$ have odd order. Let $N = T$, $S_0 = \{ab, ba\}$, $S= S_0\cup S_0^{-1}$, and let $\Gamma = \BiCay(N, \emptyset, \emptyset, S)$. Define two permutations  $\delta$ and $ \sigma $ of $V\Gamma$ where $x^\delta_\varepsilon = (x^a)_{1-\varepsilon}$, and $x^\sigma_\varepsilon = (x^a)_{\varepsilon}$ for $\varepsilon \in \{0,1\}$, and  set $G := N \rtimes\langle \sigma, \delta\rangle$.
\end{Construction} 
\begin{remark}
	For an explicit example of a simple group $T$ and generating set $\{a,b\}$ as in Construction \ref{Nonabelian k=1} take $T$ to be the alternating group Alt($n$) for odd $n\geq 5$, and let $a = (12)(34)$ and $b = (12 \dots n)$.
\end{remark}
\begin{Lemma}\label{k=1lemma}
	For $\Gamma ,G$ as in Construction \ref{Nonabelian k=2}, $(\Gamma, G) \in \OG(4)$ and is basic of biquasiprimitive type with $\soc(G)$ as described in Theorem \ref{MainResult} case (b) with $k = 1$. 
\end{Lemma}
\begin{proof}
	Since $N$ is nonabelian and the order of $b$ is odd, it follows that $S_0 \cap S_0^{-1} = \emptyset$ and hence that $|S|= 4$ and $\Gamma$ is 4-valent. Again, using the fact that $b$ has odd order it is easy to check that $a,b \in \langle S \rangle$ and hence that $\langle S \rangle = N$. Now consider $S^2$. This set contains the elements $abab, b^2$ and $baba$. In particular, $\langle S^2 \rangle$ contains $b$ and hence also contains $aba$. Since $aba$ and $abab$ are contained in $\langle S^2 \rangle$ and the order of $ab$ is odd, it follows that $a \in \langle S^2 \rangle $ and hence $ \langle S^2 \rangle  = N$. Therefore $\Gamma $ is connected.
	
	Next, notice that both $\sigma $ and $\delta$ are induced by conjugation by $a$ in $N$ and this automorphism fixes $S$ setwise. Hence $\sigma$ and $\delta$ are automorphisms of $\Gamma$ by Proposition \ref{AutBiCay}. The stabilizer of the vertex $(1_N)_0$ is $\langle \sigma \rangle$ with two orbits on the neighbours of $(1_N)_0$, namely $\{(ab)_1, (ba)_1\}$ and $\{(b^{-1}a)_1, (ab^{-1})_1\}$.  Furthermore a straightforward check shows that the only automorphisms in $G$ mapping $1_0$ to $(ab)_1$ are $g_1 = (ab)\sigma\delta$ and $g_2 = (ba)\delta$ (where $(ab)$ and $(ba)$ are automorphisms contained in $N$) and neither of these map $(ab)_1$ to $1_0$.
	This implies that $\Gamma$ is $G$-oriented and hence that $(\Gamma, G) \in \OG(4)$. 
	
	Now notice that neither $\langle \sigma \rangle$ nor $\langle \delta \rangle$ is normal in $G$. On the other hand, $N$ is a normal (and hence minimal normal) subgroup of $G$, and is the unique such subgroup. Since $N$ clearly has two orbits on $V\Gamma$, it follows that $G$ is biquasiprimitive on the vertices of $\Gamma$.
\end{proof}


\begin{Construction}\label {Nonabelian k=2} Let $T$ be a nonabelian simple group, and let $\{a,b\}$ be a generating set for $T$ such that no automorphism of $T$ swaps $a$ and $b$, and the elements $a$ and $b$ have odd order. Let $N = T\times T$, $S_0 = \{(a,b),(b,a)\}$, $S= S_0\cup S_0^{-1}$, and let $\Gamma = \BiCay(N, \emptyset, \emptyset, S)$. Define two permutations  $\delta$ and $\sigma$ of $V\Gamma$ where $(x,y)^\delta_\varepsilon = (y,x)_{1-\varepsilon}$, and $(x,y)^\sigma_\varepsilon = (y,x)_{\varepsilon}$ for $\varepsilon \in \{0,1\}$. Set $G := N \rtimes\langle \sigma, \delta\rangle$.
	
\end{Construction} 
\begin{remark}
	For an explicit example of a simple group $T$ and generating set $\{a,b\}$ as in Construction \ref{Nonabelian k=2} take $T$ to be the alternating group Alt($n$) for odd $n$, and let $a = (123)$ and $b = (12 \dots n)$.
\end{remark}

\begin{Lemma}\label{k=2lemma}
	For $\Gamma ,G$ as in Construction \ref{Nonabelian k=2}, $(\Gamma, G) \in \OG(4)$ and is basic of biquasiprimitive type with $\soc(G)$ as described in Theorem \ref{MainResult} case (b) with $k = 2$. 
\end{Lemma}

\begin{proof}
	
	First notice that $S_0\cap S_0^{-1} = \emptyset$ 
	since if $(a,b)^{-1} = (a,b)$, then both $a$ and $b$ are involutions, while if $(a,b)^{-1} = (b,a)$ then $\langle a,b \rangle = \langle a \rangle$ is cyclic, and neither of these is possible. In particular, $|S| = 4$ and $\Gamma$ is 4-valent.
	
	To see that $\Gamma$ is connected consider the following. The projections of $\langle S\rangle $ onto the simple direct factors of $N = T\times T$ are both equal to the group $\langle a, b \rangle = T$. Hence either $\langle S \rangle = N$ or $\langle S\rangle  = \{(t,t^\varphi), t\in T\}$ for some $\varphi \in \Aut(T)$. In the latter case, $(a,b) = (a, a^\varphi)$ so $b= a^\varphi$, but also $(b,a) = (b, b^\varphi)$ so $a = b^\varphi$, but by our assumption no such automorphism $\varphi$ exists. Hence $N = \langle S\rangle$. Finally, notice that since both $a$ and $b$ have odd order, we have $(a,b) \in \langle (a^2,b^2) \rangle$ (and similarly $(b,a) \in \langle (b^2,a^2) \rangle$). In particular both $(a,b)$ and $(b,a)$ are contained in $\langle S^2 \rangle $, so $N = \langle S\rangle = \langle S^2 \rangle $, and $\Gamma$ is connected.
	
	Once again Proposition \ref{AutBiCay} implies that $\sigma, \delta \in \Aut(\Gamma)$. Now it is clear that $G$ acts transitively on the vertices of $\Gamma$ and the stabilizer of the vertex $(1_N)_0$ is exactly $\langle \sigma \rangle \cong C_2$ with two orbits on the neighbourhood of $(1_N)_0$. Moreover, it is easy to check that no automorphism can reverse edges as follows. The only automorphisms taking $(1_N)_0$ to $(a,b)_1$ are $g_1 = n_1\sigma\delta$ and $g_2 = n_2\delta$ where $n_1 = (a,b)$ and $n_2 = (b,a)$ are elements of $N$. Since neither of these maps $(a,b)_1$ to $(1_N)_0$, it follows that $\Gamma$ is $G$-oriented and $(\Gamma, G) \in \OG(4)$.
	
	Finally, since conjugation by $\sigma$ in $G$ interchanges the two simple direct factors of $N$, it follows that $N$ is a minimal normal subgroup of $G$ and so is the unique minimal normal subgroup. Of course, $N$ has two orbits on $V\Gamma$, thus $G$ is biquasiprimitive on the vertices of $\Gamma$.
\end{proof}

Next we give a construction of biquasiprimitive basic pairs as described in Theorem \ref{MainResult} case (b) with $k = 4$. This time we will use Method \ref{CosetMethod}. We will use the same simple group $T$ and generating pair $\{a,b\}$ in Constructions \ref{Nonabelian k=4}, \ref{Nonabelian ell=2} and \ref{Nonabelian ell=4}. Hence we begin with the following important remark.

\begin{remark}\label{PSLgen}
	For a prime $p\geq 7$ let $T$ denote the simple group $\PSL(2,p)$. Then $T$ is generated by two elements $a$ and $b$ where 
	$$a:=\begin{pmatrix} 0 & 1\\-1 & 0  \end{pmatrix} \hbox{ and } b:= \begin{pmatrix} 0 & 1\\-1 & 1  \end{pmatrix}.$$
	Moreover $a$ and $b$ have orders 2 and 3 respectively, while $ab$ and $ab^2$ have order $p$, \cite[Section 7.5]{coxeter2013generators}.
\end{remark}

\begin{Construction}\label{Nonabelian k=4} For a prime $p \geq 7$ let $T$ denote the simple group $\PSL(2,p)$ generated by two elements $a$ and $b$ such that $a$ and $b$ have orders $2$ and $3 $ respectively while  $ab$ and $ab^2$ have order $p$.
	Take the group $T \wr S_4$ with $S_4$ acting by permuting the four direct factors of $T^4$ and define the following elements of this group  $$ \tilde{\varphi} := (b,ba,ab,aba)(13),$$
	$$y:= \tilde{\varphi}^2  = (bab,baba,ab^2,ab^2a),$$
	$$h_1 := (a,a,a,a)(12)(34),$$
	$$h_2 := h_1^{\tilde{\varphi}} = (b^{-1}aba,ab^{-1}ab,b^{-1}aba,ab^{-1}ab)(14)(23).$$
	Now let $V := \langle h_1, h_2 \rangle $ and define the subgroup $H := T^4 \rtimes V \leq T \wr S_4$. Notice that conjugation by $\tilde{\varphi}$ in  $T \wr S_4$ induces an automorphism $\varphi \in \Aut(H)$, in particular $\varphi^2$ is the inner automorphism of $H$ corresponding to conjugation by $y \in H$.
	
	Finally apply Construction \ref{GeneralCoset} using $H, V, y $ and $ \varphi$ to get the pair $(\Gamma, G)$.
\end{Construction} 

\begin{Lemma}
	Let $\Gamma ,G$ be as in Construction \ref{Nonabelian k=4}. Then $(\Gamma, G) \in \OG(4)$ and is basic of biquasiprimitive type with $\soc(G)$ as described in Theorem \ref{MainResult} case (b) with $k = 4$.
\end{Lemma}
\begin{proof}
	Since Construction \ref{Nonabelian k=4} is a special case of Construction \ref{GeneralCoset}, in order to show that $(\Gamma, G) \in \OG(4)$ it suffices to show that condition (2) of Lemma \ref{cond2} is satisfied.
	
	First notice that $V \cong \mathbb{Z}_2^2$ since $h_1$ and $h_2$ are commuting involutions. Also $V$ is core-free in $H$ since for instance $V\cap V^y = 1$. It is also easy to check that $V\cap V^\varphi = \langle h_2\rangle$, and so $|V: V\cap V^\varphi| = 2$.
	
	Now suppose that $y \in VV^\varphi$ so that $y = vu$ for some $v\in V,$ and $ u \in V^\varphi$. This implies that $vu\in T^4$, meaning that if we take $\pi$ to be the projection map $H \rightarrow S_4$, then $\pi(v) = \pi(u)$. Hence the only possibilities for $(v,u)$ such that $y = vu$ that need to be considered are $(h_1,h_2^\varphi), (h_2,h_2),$ and $  (h_1h_2,h_2h_2^\varphi)$, the second possibility gives $h_2^2 = 1 \neq y$, while the first and third of these possibilities both give $y = h_1h_2^\varphi$. It is easy to check however that $h_1h_2^\varphi = h_1h_1^y$ has $bab^2$ in its third coordinate while $y$ has $ab^2$ in its third coordinate. Hence $y \notin VV^\varphi$.
	
	Left to show is that  $\langle V, y\rangle = H$.
	To prove this claim it is sufficient to show that $T^4 \leq \langle V, y\rangle$. To this end, let $y_1 := y^{h_1}$ and $y_2 := y^{h_2}$, so that we have 
	\begin{align*}
	y &= (bab,baba,ab^2,ab^2a), \\
	y_1  &= (abab,ababa,b^2,b^2a)\hbox{, and} \\
	y_2  &=  (b^2ab^2ab,ab^2abababa,b^2abab^2,ab^2).
	\end{align*}
	We will show that $T^4 = \langle y,y_1,y_2\rangle \leq \langle V,y\rangle$. 
	
	First, it is straightforward to check that the group $\langle y,y_1,y_2\rangle$ projects onto each direct factor of $T^4$. Consider now the elements of $T$  appearing as coordinates of $y, y_1$ and $y_2$. It is easy to see that the three elements $ab^2a,$ $ b^2,$ and $b^2ab^2ab$ have order 3. On the other hand, using the fact that $ab$ and $ab^2$ have order $p$, we can check that $abab, bab$ and $b^2abab^2$ also have order $p$. The remaining elements appearing as coordinates of $y, y_1$ and $y_2$ are conjugates of these elements of order $p$ and hence also have the same order. In particular, since the only elements of order 3 ($ab^2a, b^2$, and $b^2ab^2ab$), appear in the fourth, third and first coordinates of $y, y_1$ and $y_2$ respectively, and $\langle y,y_1,y_2\rangle$ is a subdirect subgroup of $T^4$, it follows that $T^4 = \langle y,y_1,y_2\rangle$ and so $\langle V, y\rangle = H$. So by Lemma \ref{cond2}, $(\Gamma, G) \in \OG(4)$.
	
	Finally we show that $(\Gamma, G)$ is a biquasiprimitive basic pair.
	Since $H$ acts transitively on the simple direct factors of $T^4$, it follows that $T^4$ is a minimal normal subgroup of $H$, and is the unique such subgroup. Hence $G^+$ has a unique minimal normal subgroup $N = \Diag_\varphi(T^4 \times T^4) \cong T^4$ and this must be the unique minimal normal subgroup of $G$. Hence $(\Gamma, G)$ is biquasiprimitive by Corollary \ref{CorMin}.
\end{proof}

We conclude this section by giving constructions of basic biquasiprimitive ($\Gamma, G) \in \OG(4)$ as described in Theorem \ref{MainResult} case (c). The first construction is similar to Construction \ref{Nonabelian k=2}. As in that construction, the alternating group Alt($n$) with $n$ odd, and generators $a = (123)$ and $b = (1...n)$ will have the required properties.

\begin{Construction}\label{Nonabelian ell=1} 
	Let $T$ be a nonabelian simple group, and let $\{a,b\}$ be a generating set for $T$ such that no automorphism of $T$ swaps $a$ and $b$, and the elements $a$ and $b$ have odd order. Suppose further that there is an automorphism $\theta \in \Aut(T)$ which inverts both generators $a$ and $b$. Let $N = T\times T$, $S_0 = \{(a,b),(b,a)\}$, $S= S_0\cup S_0^{-1}$, and let $\Gamma = \BiCay(N, \emptyset, \emptyset, S)$. Define two permutations $\delta$ and $\sigma$ of $V\Gamma$,  where $(x,y)^\delta_\varepsilon = (y,x)_{1-\varepsilon}$, and $(x,y)^\sigma_\varepsilon = (x^{\theta},y^{\theta})_{\varepsilon}$ for $\varepsilon \in \{0,1\}$. Set $G := N \rtimes\langle \sigma, \delta\rangle$.	
\end{Construction} 
\begin{Lemma}
	For $\Gamma ,G$ as in Construction \ref{Nonabelian ell=1}, $(\Gamma, G) \in \OG(4)$ and is basic of biquasiprimitive type with $\soc(G)$ as described in Theorem \ref{MainResult} case (c) with $\ell = 1$. 
\end{Lemma}
\begin{proof}
	Since $\Gamma$ is the same graph from Construction \ref{Nonabelian k=2}, it follows from Lemma \ref{k=2lemma} that $\Gamma$ is 4-valent and connected. Again $\sigma$ and $\delta$ are induced by automorphisms of $N$ which fix $S$ and hence are automorphisms of $\Gamma$ by Proposition \ref{AutBiCay}. Moreover it is a straightforward check that the stabilizer of the vertex $(1_N)_0$ is $\langle\sigma\rangle \cong C_2$ and also that there are only two automorphisms in $G$ mapping the vertex $(1_N)_0$ to its neighbour $(a,b)_1$ but neither of these reverses the edge $\{(1_N)_0,(a,b)_1\}$. Hence $\Gamma$ is $G$-oriented and  $(\Gamma, G) \in \OG(4)$.
	
	Now notice the setwise stabilizer of $\Delta : = N_0$ is $G^+ = N\langle \sigma \rangle$ and  that $T\times 1 \leq N$ is a normal subgroup of $G^+$ which is intransitive on  $\Delta$. In particular, $G^+$ is not quasiprimitve on $\Delta$. Moreover $\delta$  interchanges the two simple direct factors of $N$, and hence $N$ is the unique minimal normal subgroup of $G$. Since $N$ is contained in $G^+$ it follows that $\Gamma$ is basic of biquasiprimitive type as in Theorem \ref{nonabelianBound} case (b).
\end{proof}	

The next two constructions both provide pairs $(\Gamma, G)$ as described in Theorem \ref{MainResult} case (c) with $\ell = 2$ and $\ell = 4$ respectively. In both cases $\soc(G) = T^{2\ell}$ where $T$ is the simple group $\PSL(2,p)$. In both cases we may use the same generating pairs $\{a,b\}$ as those used in Construction \ref{Nonabelian k=4} (see Remark \ref{PSLgen}).

\begin{Construction}\label{Nonabelian ell=2} For a prime $p \geq 7$ let $T$ denote the simple group $\PSL(2,p)$ generated by two elements $a$ and $b$ such that $a$ and $b$ have orders $2$ and $3$ respectively while  $ab$ and $ab^2$ have order $p$.
	Take the group $T \wr S_4$ with $S_4$ acting by permuting the four direct factors of $T^4$ and define the following elements of this group  	$$\tilde{\varphi} := (b^2ab, ab^2,b^2, a)(13)(24),$$
	$$y := \tilde{\varphi}^2 = (b^2a, ab^2a, bab, b^2),$$
	$$h_1 := (a,a,a,a)(12)(34).$$
	
	Now let $V := \langle h_1 \rangle $ and define the subgroup $H := T^4 \rtimes V \leq T \wr S_4$. Notice that conjugation by $\tilde{\varphi}$ in  $T \wr S_4$ induces an automorphism $\varphi \in \Aut(H)$, in particular $\varphi^2$ is the inner automorphism of $H$ corresponding to conjugation by $y \in H$.
	
		Finally apply Construction \ref{GeneralCoset} using $H, V, y $ and $ \varphi$ to get the pair $(\Gamma, G)$.
\end{Construction} 

\begin{Lemma}
	Let $\Gamma ,G$ be as in Construction \ref{Nonabelian ell=2}. Then $(\Gamma, G) \in \OG(4)$ and is basic of biquasiprimitive type with $\soc(G)$ as described in Theorem \ref{MainResult} case (c) with $\ell=2$.
\end{Lemma}
\begin{proof}
	Again, since this is a  special case of Construction \ref{GeneralCoset}, in order to show that $(\Gamma, G) \in \OG(4)$ it is sufficient to show that condition (2) of Lemma \ref{cond2} is satisfied.
	
	Here $V \cong C_2$ and since $h_1^\varphi \notin V$ we have that $V$ is core-free in $H$ and $|V:V\cap V^\varphi| = 2$. It is also easy to check that  $y \notin VV^\varphi$ in this case, by noticing that  $y \neq h_1h_1^\varphi$.
	
	Left to show is that $\langle V, y \rangle = H$. In fact, we will show that $T^4\leq \langle y_1, y\rangle$ where $y_1 := y^{h_1} = (b^2, ab^2, ab^2a, ababa)$, from which it follows that $\langle V, y \rangle = H$. First, $\langle y, y_1 \rangle $ projects onto each factor of $T^4$ so we only need to make sure that $\langle y, y_1 \rangle $ is not a product of diagonal subgroups of $T^4$. Now $y$ has elements of order $p$ in its first and third coordinates and elements of order 3 in its second and fourth coordinates. Thus all we need to check is that no automorphism of $T$ can map $b^2a$ to $bab$ and $b^2$ to $ab^2a$ and that no automorphism can map $ab^2a$ to $b^2$ and $ab^2$ to $ababa$. In the first case, such an automorphism must map $a$ to $(ab)^3$ which is impossible since $a$ is an involution. In the second case, such an automorphism must map $a$ to $(ab^2)^3$, which again is not an involution. Hence $T^4\leq \langle y_1, y\rangle$ and so $\langle V, y \rangle = H$. By Lemma \ref{cond2} this gives $(\Gamma, G) \in \OG(4)$.
	
	It is clear that the action of $H$ on the simple direct factors of $T^4$ has two orbits of length 2. Thus $H$ has two minimal normal subgroups isomorphic to $T^2$, and these are the only minimal normal subgroups of $H$. Furthermore, it is clear that the automorphism $\varphi$ of $H$ interchanges these normal subgroups. We will thus let $R$ and $R^\varphi$ denote these two minimal normal subgroups of $H$.
	
	Since $G^+ \cong H$, $G^+$ also has two minimal normal subgroups isomorphic to $R, R^\varphi \cong T^2$. Let $K$ and $L$ denote these minimal normal subgroups of $G^+$ so $K = \Diag_\varphi(R\times R)$ and $L = \Diag_\varphi(R^\varphi \times R^\varphi)$. Then conjugation by $g$ in $G$, interchanges $K$ and $L$ and so $G$ acts transitively on the direct factors of $\soc(G^+) = K\times L \cong T^4$. Hence $\soc(G^+)$ is a minimal normal subgroup of $G$ and $(\Gamma, G)$ is biquasiprimitive by Corollary \ref{CorMin}.
\end{proof}	

\begin{Construction}\label{Nonabelian ell=4} For a prime $p \geq 7$ let $T$ denote the simple group $\PSL(2,p)$ generated by two elements $a$ and $b$ such that $a$ and $b$ have orders $2$ and $3$ respectively while  $ab$ and $ab^2$ have order $p$.
	Take the group $T \wr S_8$ with $S_8$ acting by permuting the eight direct factors of $T^8$ and define the following elements of this group  $$\tilde{\varphi} := (b,ba,ab,aba,b^2,ab,ba,ab^2a)(15)(28)(37)(46),$$
	$$y := \tilde{\varphi}^2 = (1,a,ab^2a,ab^2,1,ababa,b^2,ab^2aba),$$
	$$h_1 := (a,a,a,a,a,a,a,a)(12)(34)(56)(78),$$
	$$h_2 := h_1^{\tilde{\varphi}} = (b^2,ab^2a,aba,b,b^2aba,ab^2ab,b^2aba,ab^2ab)(14)(23)(58)(67).$$
	
	Now let $V := \langle h_1 ,h_2\rangle $ and define the subgroup $H := T^8 \rtimes V \leq T \wr S_8$. Notice that conjugation by $\tilde{\varphi}$ in  $T \wr S_8$ induces an automorphism $\varphi \in \Aut(H)$, in particular $\varphi^2$ is the inner automorphism of $H$ corresponding to conjugation by $y \in H$.
	
	Finally apply Construction \ref{GeneralCoset} using $H, V, y $ and $ \varphi$ to get the pair $(\Gamma, G)$.
\end{Construction} 
\begin{Lemma}\label{lastConstructionLemma}
	Let $\Gamma ,G$ be as in Construction \ref{Nonabelian ell=4}. Then $(\Gamma, G) \in \OG(4)$ and is basic of biquasiprimitive type with $\soc(G)$ as described in Theorem \ref{MainResult} case (c) with $\ell=4$.
\end{Lemma}

\begin{proof}
	As in previous constructions, we only need to check that condition (2) of Lemma \ref{cond2} is satisfied to show that $(\Gamma, G) \in \OG(4)$.
	Here we have $V = \langle h_1, h_2 \rangle \cong \mathbb{Z}_2^2$ with $V\cap V^\varphi = \langle h_2\rangle$ and $V\cap V^y = 1$. Hence $V$ is core-free in $H$ and $|V: V\cap V^\varphi| = 2$. To check that $y \notin VV^\varphi$ it is sufficient to check that $y \neq h_1h_1^y$ and this is clearly true since $y\neq 1$.  
	
	Left to show is that $\langle V, y \rangle = H$. Of course, it is sufficient to show that $T^8 \leq \langle V, y \rangle$, and in fact we will show that $T^8 = \langle y, y_1,y_2\rangle \leq \langle V, y\rangle$ where $y_1:= y^{h_1}$, and $y_2 := y^{h_2}$. Hence we have
	
	$$ y = (1,a,ab^2a,ab^2,1,ababa,b^2,ab^2aba), $$
	$$y_1 = (a,1,b^2a,b^2,bab,1,b^2ab,ab^2a),\hbox{ and }$$ 
	$$y_2 = (b^2a, ab^2a,abab^2a,1,b^2ab,ab^2ab^2aba,b^2a,1). $$
	
	It is easy to check that $\langle y,y_1,y_2\rangle$ projects onto each direct factor of $T^8$. Furthermore, notice that the identity element occurs in the first and fifth coordinates of $y$, the second and sixth coordinates of $y_1$, and the fourth and eighth coordinates of $y_2$. So if $\langle y,y_1,y_2\rangle$ is a product of diagonal subgroups of $T^8 = \Pi _{i =1} ^8 T_i$  then each direct factor of  $\langle y,y_1,y_2\rangle$ must be either a full subgroup $T_j$ for some  $1\leq j\leq 8$, or is a diagonal subgroup of a subproduct $T_m \times T_n$ where $(m,n) \in \{(1,5),(2,6),(3,7),(4,8)\}$.
	
	However, the elements in the first and fifth coordinates of $y_2$ have orders $p$ and 2 respectively, the elements in the second and sixth coordinates of $y$ have orders 2 and $p$ respectively, the elements in the third and seventh coordinates of $y_1$ have orders $p$ and 2 respectively, and the elements in the fourth and eighth coordinates of $y$ have orders $p$ and 2 respectively. 
	Therefore  $\langle y, y_1, y_2\rangle = \Pi _{i =1} ^8 T_i = T^8$ and hence $H = \langle V, y\rangle$.  Lemma \ref{cond2} now implies that $(\Gamma, G) \in \OG(4)$.
	
	Notice that the action of $H$ on the simple direct factors of $T^8$ has two orbits of length 4. Thus $H$ has two minimal normal subgroups isomorphic to $T^4$, and these subgroups are interchanged by the automorphism $\varphi \in \Aut(H)$. As in previous constructions, we let $R$ and $R^\varphi$ denote these two minimal normal subgroups of $H$.
	
	Since $G^+ \cong H$, $G^+$ also has two minimal normal subgroups, namely $K = \Diag_\varphi(R\times R)$ and $L = \Diag_\varphi(R^\varphi \times R^\varphi)$, and conjugation by $g \in G$ interchanges $K$ and $L$, implying that $G$ acts transitively on the direct factors of $K\times L \cong T^8$. In particular, $\soc(G^+) = K \times L$ is a minimal normal subgroup of $G$, and hence $(\Gamma, G)$ is biquasiprimitive by Corollary \ref{CorMin}. This shows that $(\Gamma, G)$ is a basic biquasiprimitive pair as described in Theorem \ref{nonabelianBound} case(b) with $\ell=4$.
\end{proof}

Constructions \ref{Abelian k=1} - \ref{Nonabelian ell=4} together with Lemmas \ref{firstConstructionLemma} - \ref{lastConstructionLemma}, and the remarks in this section which give explicit simple groups and generating pairs for each construction, complete the proof of Theorem \ref{MainResult}. 
Note that in each of the explicit examples of biquasiprimitive pairs ($\Gamma, G)$ provided here, the group $G$ contains a subgroup $N$ acting semi-regularly with two orbits on $V\Gamma$, hence all of these examples are bi-Cayley graphs. Of course, it should not be too difficult to construct non-bi-Cayley examples using Method \ref{CosetMethod}.

An interesting further question would be to determine which nonabelian simple groups $T$ can occur as the simple direct factors of the socle of $G$ where $(\Gamma, G)\in \OG(4)$ is biquasiprimitive.

\subsection*{Acknowledgements}
We are grateful for the opportunity to participate in the Tutte Memorial MATRIX retreat where this work began. The first author was supported by an Australian Government Research Training Program (RTP) Scholarship, and the second author acknowledges Australian Research Council Funding of DP160102323.
We acknowledge the use of \textsc{Magma} \cite{MR1484478} for testing theories and constructing examples.
	\bibliographystyle{abbrv}
	\bibliography{Biquasiprimitive}
\end{document}